\documentclass[11pt]{amsart}

\usepackage{amsmath}
\usepackage{relsize}
\usepackage[noadjust]{cite}
\usepackage{color}
\usepackage[dvipsnames]{xcolor}
\usepackage{mathtools}
\usepackage[normalem]{ulem}
\newcommand\redout{\bgroup\markoverwith
{\textcolor{red}{\rule[.5ex]{2pt}{0.4pt}}}\ULon}

\usepackage{hyperref, enumitem}
\hypersetup{
  colorlinks   = true, 
  urlcolor     = blue, 
  linkcolor    = Purple, 
  citecolor   = red 
}
\usepackage{amsmath,amsthm,amssymb}

\usepackage[margin=2.5cm]{geometry}
\usepackage{caption} 

\usepackage{tikz-cd}
\usetikzlibrary{calc,fit,matrix,arrows,automata,positioning}
\usetikzlibrary{decorations.pathreplacing,angles,quotes}

\usetikzlibrary{fit}
\tikzset{%
  highlight/.style={rectangle,rounded corners,fill=red!15,draw,
    fill opacity=0.5,thick,inner sep=0pt}
}

\usepackage{array}
\newcolumntype{x}[1]{>{\centering\arraybackslash\hspace{0pt}}p{#1}}

\DeclareMathOperator{\C}{\mathcal{C}}

\DeclareMathOperator{\rk}{rk}

\newtheorem{theorem}{Theorem}[section]

\newtheorem{corollary}[theorem]{Corollary}
\newtheorem{definition}[theorem]{Definition}
\newtheorem{proposition}[theorem]{Proposition}

\newtheorem{example}[theorem]{Example}

\newtheorem{remark}[theorem]{Remark}
\newtheorem{property}[theorem]{Property}

\newtheorem{question}[theorem]{Question}

\newcommand{\mM}{\mathcal{M}}
\newcommand{\mL}{\mathcal{L}}
\newcommand{\mU}{\mathcal{U}}

\newcommand{\mS}{\mathcal{S}}

\newcommand{\Fmk}{[n,k]_{q^m/q}}

\newcommand{\fqm}{\mathbb{F}_{q^m}}

\newcommand{\F}{{\mathbb F}}

\newcommand{\GL}{\hbox{{\rm GL}}}

\newcommand{\fq}{{\mathbb F}_{q}}

\newcommand{\la}{\langle}
\newcommand{\ra}{\rangle}

\newcommand{\N}{\mathrm{N}}

\newcommand{\Fm}{\mathbb{F}_{q^m}}

\title{Delsarte duality on subspaces and applications to rank-metric codes and $q$-matroids}
\usepackage[foot]{amsaddr}
\author{Martino Borello$^{1,2}$}
\author{Olga Polverino$^3$}
\author{Ferdinando Zullo$^3$}
\address{$^1$Universit\'e Paris 8, Laboratoire de G\'eom\'etrie, Analyse et Applications, LAGA, Universit\'e Sorbonne Paris Nord, CNRS, UMR 7539, France.}
\address{$^2$Inria Saclay, France.}
\address{$^3$Dipartimento di Matematica e Fisica, Universit\`a degli Studi della Campania ``Luigi Vanvitelli'', I--\,81100 Caserta, Italy}

\email{martino.borello@univ-paris8.fr, \{olga.polverino,ferdinando.zullo\}@unicampania.it}

\begin{document}
\maketitle

\begin{abstract}
We study the interplay between the lattice of $\mathbb{F}_{q^m}$-subspaces and the lattice of $\mathbb{F}_q$-subspaces of an $\mathbb{F}_{q^m}$-vector space. Introducing notions of weight and defect relative to an $\mathbb{F}_q$-subspace, we analyze the sequence of maximum non-zero defects. 
We establish a correspondence between subspaces of positive  defect and their Delsarte duals, enabling explicit characterizations of the associated sequences of maximum non-zero defects. Our framework unifies several classes of subspaces studied in finite geometry and connects them to linear rank-metric codes by providing a new geometric interpretation of code duality.
Building on these results, we characterize classes of rank-metric codes closed under duality, including MRD, near MRD, quasi-MRD, and a new family of $(\mathbf{n}, \mathbf{k})$-MRD codes. Finally, we explore applications to $q$-matroids, by studying the problem of $\mathbb{F}_{q^m}$-representability for direct sums of uniform $q$-matroids and describing their rank generating functions.
\end{abstract}



\tableofcontents

\bigskip

\section*{Introduction}

Consider an $\F_{q^m}$-vector space $\mathbb{V}(k, q^m)$ of dimension $k$. Clearly, $\mathbb{V}(k, q^m)$ can also be viewed as an $\F_q$-vector space of dimension $mk$. Therefore, we can consider the lattice of $\F_{q^m}$-subspaces of $\mathbb{V}(k, q^m)$, denoted by ${\mathcal L}(\mathbb{V}(k, q^m))$, and the lattice of $\F_q$-subspaces of $\mathbb{V}(k, q^m)$, denoted by ${\mathcal L}_q(\mathbb{V}(k, q^m))$.
This paper focuses on the interplay between the two lattices ${\mathcal L}(\mathbb{V}(k, q^m))$ and ${\mathcal L}_q(\mathbb{V}(k, q^m))$.
Such an area of research is certainly not new, especially within finite geometry: the classical notion of subgeometry can be rephrased in this terminology. Indeed, a subgeometry can be defined as the set of points defined by vectors of an element $U \in {\mathcal L}_q(\mathbb{V}(k, q^m))$ of dimension $k$, with the property that every element $T \in {\mathcal L}(\mathbb{V}(k, q^m))$ intersects $U$ in an $\F_q$-subspace of dimension at most $\dim_{\F_{q^m}}(T)$. This concept led to the theory of $\F_q$-linear sets, which has evolved into a vibrant and widely studied area of research. Over the years, $\F_q$-linear sets have found numerous applications in different contexts of finite geometry and coding theory, including the construction of blocking sets in projective spaces, the study of semifield flocks, translation ovoids, and, more recently, in the theory of rank-metric codes; see e.g. \cite{polverino2010linear,lavrauw2015field}. However, we will describe the context, results, and connections within the vectorial framework, as this appears to be the most suitable for the links with rank-metric codes and $q$-matroids that we will explore in this paper. 

Two central notions will play a key role in our study: the weight and the defect of the elements in ${\mathcal L}(\mathbb{V}(k, q^m))$ with respect to a subspace $U$ in ${\mathcal L}_q(\mathbb{V}(k, q^m))$.
Intuitively, the weight of a subspace $T \in {\mathcal L}(\mathbb{V}(k, q^m))$ with respect to $U$ measures how much of $U$ is \emph{concentrated} in $T$, whereas the defect of $T$ measures how far $T \cap U$ is from being a subgeometry of $T$. These two notions are closely related, as we shall see.
The main focus is the study of the weights and defects of subspaces in ${\mathcal L}(\mathbb{V}(k, q^m))$ (of any dimension) with respect to $U$. This naturally leads us to define, for each fixed dimension, the maximum defect that  elements of ${\mathcal L}(\mathbb{V}(k, q^m))$ can have with respect to $U$.
We prove that the sequence of maximum defects is non-decreasing with respect to the dimension. Therefore, it suffices to consider the restriction of this sequence to the strictly increasing subsequence, which we call the sequence of maximum non-zero defects.
After studying some basic properties of this sequence, we recall the notion of the Delsarte dual, introduced in \cite{csajbok2021generalising} in connection with rank-metric codes.
We investigate the relationship between the sequences of maximum non-zero defects of a subspace $U \in \mathcal{L}_q(\mathbb{V}(k,q^m))$ and its Delsarte dual $U^d$.
This connection involves a detailed analysis of the subspaces of $\mathcal{L}(\mathbb{V}(k,q^m))$ having positive defect with respect to $U$ and those of $\mathcal{L}(\mathbb{V}(\dim{\fq}(U)-k,q^m))$ having positive defect with respect to $U^d$.
In particular, we establish a one-to-one correspondence between the subspaces in $\mathcal{L}(\mathbb{V}(k,q^m))$ having positive defect with respect to $U$ and minimal with respect to their defect, and the subspaces in $\mathcal{L}(\mathbb{V}(\dim_{\fq}(U)-k,q^m))$ having positive defect with respect to $U^d$ and minimal with respect to their defect.
Thanks to this correspondence, we are able to determine the maximum non-zero defects of the Delsarte dual $U^d$ in terms of the maximum non-zero defects of $U$.
As a consequence of these results, we analyze some known classes of subspaces and we study their Delsarte duals. In particular, we have considered:
\begin{itemize}
    \item Scattered subspaces with respect to $h$-subspaces, introduced in \cite{lunardon2017mrd,sheekeyVdV,csajbok2021generalising,blokhuis2000scattered};
    \item $1$-defect subspaces, including clubs (introduced in \cite{fancsali2006maximal,fancsali2009description} and in \cite{mannaert2025clubs} for larger dimensions);
    \item Decomposable subspaces, introduced in \cite{bartoli2024Exceptional_sequences,santonastaso2024completely};
    \item $\mathbf{k}$-scattered subspaces with respect to the hyperplanes, introduced in \cite{alfarano2024representability}.
\end{itemize}

This study has an immediate impact on the theory of rank-metric codes. Indeed, in \cite{Randrianarisoa2020ageometric} (see also \cite{alfarano2021linearcutting,sheekey2019scatterd}) it is shown that to any linear rank-metric code in $\F_{q^m}^n$ of dimension $k$, one can associate an element of $\mathcal{L}_q(\F_{q^m}^k)$ (and conversely) in such a way that the metric properties of the code can be studied directly from the associated subspace.
The main drawback of this connection is that there is no direct way to view the subspaces associated with a code and its dual within the same ambient space, which makes the study of duality-invariant properties from a geometric perspective difficult. Motivated by this limitation, we propose a new framework for describing the connection between linear rank-metric codes and subspaces in $\mathcal{L}_q(\F_{q^m}^k)$, which allows a unified approach to studying the systems associated with a code and its dual. The key point has been to view any subspace as the projection of a subgeometry; see \cite{lunardon2004translation}.
This framework allows us to prove that the operation of duality for linear rank-metric codes corresponds to the Delsarte duality of subspaces, previously proved only for a few cases (see \cite{csajbok2021generalising}).
In this new framework we have also shown how we can obtain previously known results for rank-metric codes, such as the description of the support of subcodes, the geometric view of generalized weights and the Wei-type duality.
Moreover, relying on the proposed geometric framework, we study classes of codes which are closed under duality, including MRD codes (already known since \cite{de78,ga85a}), near MRD codes (see \cite{marino2023evasive}), quasi-MRD codes (see \cite{marino2023evasive,de2018weight}) and $(\mathbf{n},\mathbf{k})$-MRD codes. The latter class has been introduced in this paper, and it can be defined as those direct sums of MRD codes whose codewords have weight lower bounded by the weights of an MRD code with the same parameters, unless these codewords belong to special projections. We prove that the geometric description of these codes corresponds to $\mathbf{k}$-scattered subspaces with respect to hyperplanes and we are able to prove that, in the case of two components, the weight distribution of such codes only depends on their parameters. We conclude the paper with an application to $q$-matroids. 
A $q$-matroid is the $q$-analogue of a classical matroid, where the role of subsets of a finite set is replaced by subspaces of a finite-dimensional vector space over a finite field. They provide a natural framework for studying rank-metric codes, as they capture geometric and combinatorial properties of codes in a way analogous to how classical matroids capture properties of linear block codes; see \cite{gorla2019rank, shiromoto2019codes, ghorpade2020polymatroid, byrne2022constructions, byrne2021weighted, gluesing2021q} for some references. One way to construct a $q$-matroid is via the supports of the codewords of a linear rank-metric code in $\F_{q^m}^n$. A $q$-matroid obtained in this way is called $\mathbb{F}_{q^m}$-representable. A well-studied class of $q$-matroids is the class of uniform $q$-matroids, which correspond to the $q$-matroids associated with MRD codes. Recently, in \cite{ceria2021direct} (see also \cite{gluesing2023decompositions}) the notion of the direct sum of $q$-matroids was introduced, and the question of whether the direct sum of uniform $q$-matroids is $\F_{q^m}$-representable for some $m \in \mathbb{N}$ was answered positively in \cite{alfarano2024representability}.
The problem that remains open is the determination of the smallest field over which these $q$-matroids are representable.
Using the connection between an $\fqm$-representation of the direct sum of uniform $q$-matroids and $\mathbf{k}$-scattered subspaces with respect to the hyperplanes established in \cite{alfarano2024representability}, we prove some new instances to this problem.
Finally, we also prove that rank generating function of the direct sum of uniform $q$-matroids only depends on the parameters of the $q$-matroid.

\subsection*{Organization of the paper}

In Section 1, we investigate the combinatorics of $\F_q$-subspaces of $\mathbb{V}(k,q^m)$, focusing on the notions of weight and defect with respect to a given $\F_q$-subspace. We define the sequence of maximum non-zero defects, study its basic properties, and introduce the Delsarte dual of an $\F_q$-subspace, analysing the correspondence between the sequences of the maximum defects of a subspace and its dual. We also examine the action of this duality on several notable classes of subspaces.

In Section 2, we turn to rank-metric codes and introduce a new framework linking where to study subspaces associated with a rank-metric code and its dual within the same ambient space. This geometric viewpoint enables us to reinterpret known results on supports, generalized weights, and Wei-type duality, and to characterize families of codes closed under Delsarte duality.

In Section 3, we apply our results to $q$-matroids. Using the established connection between $k$-scattered subspaces and the representability of direct sums of uniform $q$-matroids, we prove new instances of representability and show that the rank generating function of such direct sums depends only on the parameters of the $q$-matroid.

\section{Combinatorics of $\fq$-subspaces in $\mathbb{V}(k,q^m)$}

In this section, we study the behavior of the $\fq$-subspaces of $\mathbb{V}(k,q^m)$   with respect to its lattice of $\fqm$-subspaces.
This study is carried out up to the action of the general linear group $\GL(k,q^m)$ acting on the lattices ${\mathcal L}(\mathbb{V}(k,q^m))$ and ${\mathcal L}_q(\mathbb{V}(k,q^m))$.
In this setting, as we will see later, in addition to the dimension of the $\fq$-subspaces, the following integers serve as useful invariants.  

\begin{definition}
    Let $\mathbb{V}(k,q^m)$ be an $\F_{q^m}$-vector space of dimension $k$. Let $U$ be an $\fq$-subspace of $\mathbb{V}(k,q^m)$. For any $\fqm$-subspace $T$ of $\mathbb{V}(k,q^m)$ we define the \textbf{weight of} $T$ with respect to $U$ as
    \[ w_U(T)=\dim_{\fq}(U\cap T), \]
    and the \textbf{defect of} $T$ with respect to $U$ as
    \[ \varepsilon_U(T)=w_U(T)-\dim_{\fqm}(T). \]
\end{definition}

We start by describing some properties, which are easily verified.

\begin{property} \label{property:defect}
  Let $U$ be an $\fq$-subspace of $\mathbb{V}(k,q^m)$ and let $T\in {\mathcal L}(\mathbb{V}(k,q^m))$.  
  \begin{itemize}
    \item[1.] If $\langle T\cap U \rangle_{\fqm}=T$, then $\varepsilon_U(T)\geq 0$.
    \item[2.] If $\varepsilon_U(T)\geq 0$, then $\varepsilon_U(T)= \varepsilon_U(T')$ where $T'=\langle T\cap U \rangle_{\fqm} \subseteq T$.
    \item[3.] If $\Phi \in \GL(k,q^m)$ then $\dim_{\fq}(U)=\dim_{\fq}(\Phi(U))$,   $ w_U(T)=w_{\Phi(U)}(\Phi(T))$ and  $\varepsilon_U(T)=\varepsilon_{\Phi(U)}(\Phi(T))$.
    \end{itemize}
    Also, if $T_1$ and $T_2$ are in  ${\mathcal L}(\mathbb{V}(k,q^m))$, we have 
    \begin{equation} \label{eq: subadd_weight}
       w_U(T_1+T_2)\geq  w_U(T_1)+w_U(T_2)-w_U(T_1\cap T_2). 
        \end{equation}
    \begin{equation} \label{eq: subadd_defect}
       \varepsilon_U(T_1+T_2)\geq  \varepsilon_U(T_1)+\varepsilon_U(T_2)-\varepsilon_U(T_1\cap T_2).    
        \end{equation}
 Also,  (\ref{eq: subadd_weight})  and (\ref{eq: subadd_defect}) are equalities if and only if $(T_1+T_2)\cap U=(T_1\cap U)+(T_2\cap U)$.
\end{property}

It follows from $3.$ of Property \ref{property:defect} that the dimension, the distributions of  weights and  defects of the elements of ${\mathcal L}(\mathbb{V}(k,q^m))$ with respect to an $\fq$-subspace $U$,  are invariant with respect to the action of the group $\GL(k,q^m)$. 
However, in some cases,  the dimension of the $\F_q$-subspace $U$  is sufficient to uniquely identify $U$ with respect to the action of $\GL(k,q^m)$.

\begin{property}\label{prop:subgeometry}
Let $U$ be an $\fq$-subspace of $\mathbb{V}(k,q^m)$ such that 
\begin{equation}\label{eq:subgeometry}
\dim_{\fqm}(\langle U\rangle_{\fqm})=\dim_{\fq} (U),
\end{equation}
then the following holds:
\begin{itemize}
    \item[1.] $\varepsilon_U(T)\leq 0$ for any $T$ in  ${\mathcal L}(\mathbb{V}(k,q^m))$; 
    \item[2.] if $U'$ is an $\fq$-subspace of  $\mathbb{V}(k,q^m)$ such that $\dim_{\fq} (U)=\dim_{\fq}(U')$  and $\dim_{\fqm}(\langle U'\rangle_{\fqm})=\dim_{\fq}(U')$, then there exists $\Phi \in \GL(k,q^m)$ such that $\Phi(U)=U'$;
    \item[3.] let $S$ be any $\fq$-subspace  of $U$ of dimension $h$, then there exists a unique $\fqm$-subspace $S^*$ of $\mathbb{V}(k,q^m)$ of dimension $h$ such that $S=U\cap S^*$.
\end{itemize}
Also, 1.\ is equivalent to \eqref{eq:subgeometry}, so that any $\fq$-subspace $U$ of $\mathbb{V}(k,q^m)$ that satisfies  1.\ must satisfy (\ref{eq:subgeometry}). 
\end{property}

The $\fq$-subspaces described in Property \ref{prop:subgeometry} will be called {\bf $\fq$-subgeometries}, more precisely a $t$-dimensional $\fq$-subgeometry of $\mathbb{V}(k,q^m)$ is an $\fq$-subspace $W$ of  $\mathbb{V}(k,q^m)$ having dimension $t$ such that $\dim_{\fqm}(\langle W\rangle_{\fqm})=\dim_{\fq}(W)$.

Property \ref{prop:subgeometry} states that $\fq$-subgeometries are those $\fq$-subspaces with respect to which there are no element in ${\mathcal L}(\mathbb{V}(k,q^m))$ with positive defect. Generally speaking, defects with respect to an $\fq$-subspace measure how far the subspace is from being a subgeometry. So, this motivates the following definition.

\begin{definition} 
\label{def:maximum_defect} Let $U\in \mathcal{L}_q(\mathbb{V}(k,q^m))$. For any $r\in \{0,1,\dots,k\}$, let
\[\varepsilon_U(r)=\max\{\varepsilon_U(T) \, : \, T\in  {\mathcal L}(\mathbb{V}(k,q^m)) \, \mbox{and} \, \dim_{\fqm} (T)=r\},\]
i.e. $\varepsilon_U(r)$ is the maximum defect of the  $r$-dimensional subspaces of $\mathbb{V}(k,q^m)$ with respect to $U$. Note that $\varepsilon_U(0)=0$ and $\varepsilon_U(k)=\dim_{\fq}(U)-k$. 
\end{definition}

\begin{remark}
  Another way to characterize $\fq$-subgeometries is via the maximum defects. Indeed, $U$ is an $\fq$-subgeometry if and only if $\varepsilon_U(r)\leq 0$ for any $r$.  Also, if $\langle U\rangle_{\fqm}=\mathbb{V}(k,q^m)$, then $U$ is an $\fq$-subgeometry if and only if $\varepsilon_U(r)= 0$ for any $r\in \{0,1,\dots, k\}$.
\end{remark}

The maximum defects with respect to an $\fq$-subspace satisfy the following monotonicity property.

\begin{property} \label{prop: monotonicity_maxi_defect}
Let $U$ be in $\mathcal{L}_q(\mathbb{V}(k,q^m))$ with $\dim_{\fq} (U)=n$  such that $\langle U\rangle_{\fqm}=\mathbb{V}(k,q^m)$. Then

  \begin{equation} \label{eq:monotonicity_defects}
   0=\varepsilon_U(0) \leq  \varepsilon_U(1) \leq \varepsilon_U(2)\leq \cdots \leq \varepsilon_U(k)=n-k.
\end{equation}
 Hence,  if $\varepsilon_U(j) = 0$ for some $j\in  \{1,\dots,k\}$, then $\varepsilon_U(i) = 0$ for each $i\in \{1,2,\dots,j\}$.   
\end{property}
\begin{proof}
Since $\langle U\rangle_{\fqm}=\mathbb{V}(k,q^m)$, we have $n\geq k\geq 1$ and clearly $\varepsilon_U(k)=n-k\geq 0$ and $\varepsilon_U(1)\geq 0$. Now, let $i\in\{1,2,\dots,k-1\}$ and let $T$ be an $\fqm$-subspace of $\mathbb{V}(k,q^m)$, such that $\dim_{\fqm}(T)=i$ and  $\varepsilon_U(T)=\varepsilon_U(i)$.
Since $i<k$ and $\langle U\rangle_{\fqm}=\mathbb{V}(k,q^m)$, there exists $u'\in U\setminus T$. Hence, $T'=T+\langle  u'\rangle_{\fqm}$ is an $(i+1)$-dimensional $\fqm$-subspace of $\mathbb{V}(k,q^m)$ and by (\ref{eq: subadd_defect}) we have  
\[\varepsilon_U(T')\geq \varepsilon_U(T)+\varepsilon_U(\langle  u'\rangle_{\fqm})\geq \varepsilon_U(i).\]
From which we get $\varepsilon_U(i)\leq \varepsilon_U(i+1)$.   
\end{proof}

Due to the previous property, the behavior of an $\fq$-subspace with respect to the defects of the subspaces in ${\mathcal L}(\mathbb{V}(k,q^m))$ can be directly expressed by those defect that are in strictly increasing order. This yields the following definitions.

\begin{definition} \label{def:non-zero max defects}
    Let $U$ be in $\mathcal{L}_q(\mathbb{V}(k,q^m))$ with $\dim_{\fq}(U)=n$ such that $\langle U\rangle_{\fqm}=\mathbb{V}(k,q^m)$. Let $s$ be the number of distinct non-zero integers appearing in (\ref{eq:monotonicity_defects}) and let 
    
    \begin{equation} \label{eq:non-zero_defects_dimensions}
    0=t_0<t_1<t_2<..<t_s\leq k
   \end{equation}
be such that
     \begin{equation} \label{eq:non-zero_defects}
   0<  \varepsilon_U(t_1) < \varepsilon_U(t_2)< \cdots < \varepsilon_U(t_s)=n-k,
\end{equation}
and 
\begin{equation} \label{eq:non-zero_defects_Mimim_dimensions}
    \varepsilon_U(l)=\varepsilon_U(t_{i}) \,\,\, \mbox{if} \, \,\,  t_i\leq l<t_{i+1}.
   \end{equation}

The sequence of integers (\ref{eq:non-zero_defects}) will be called the {\bf sequence of maximum non-zero defects} with respect to $U$ and the integer $s \in \{1,\dots k\}$ will be called the {\bf length} of the  sequence of maximum non-zero defects.
\end{definition}

In the following, the integers $t_1$ and $\varepsilon_{U}(t_{s-1})$ (the last non-trivial defect) will play a crucial role in the duality process  described in the next section. We will refer to $\varepsilon_{U}(t_{s-1})$ as the \textbf{maximum non-trivial defect}.

\begin{remark} \label{rem: t_i minimum dimension}
With the notation of the above definition, note that if $\varepsilon=\varepsilon_U(t_i)$ for some of the integers $t_i\in \{1,\dots,k\}$ appearing in (\ref{eq:non-zero_defects_dimensions}), then from (\ref{eq:non-zero_defects_Mimim_dimensions})
\[t_i=\min\{\dim_{\fqm} (T) \,:\, \varepsilon_U(T)=\varepsilon\}\]
and if $T$ is an $\F_{q^m}$-subspace of dimension $t_i$ such that $\varepsilon_U(T)=\varepsilon$, then  $T$ is minimal with respect to its defect, i.e. for every $T' <_{\fqm} T$ we have that $\varepsilon_{U}(T')< \varepsilon$.
\end{remark}

In the following section, subspaces that are minimal with respect to their defect will play a key role in describing the behavior of an $\fq$-subspace; therefore, it is necessary to provide a formal definition of such subspaces.

\begin{definition}
 Let $U$ be in $\mathcal{L}_q(\mathbb{V}(k,q^m))$ and let $T$ an $\fqm$-subspace of $\mathbb{V}(k,q^m)$. We say that  $T$  is \textbf{minimal} with respect to its defect (with respect to $U$) if $T' <_{\fqm} T$ implies  $\varepsilon_{U}(T')< \varepsilon_{U}(T)$. Note that from the minimality of $T$ we also get $\la T\cap U\ra_{\fqm}=T$ and hence, by $1.$ of Proposition \ref{property:defect}, if $T\neq \{\underline 0\}$,  we have  $\varepsilon_U(T)>0$.
\end{definition}

Note that the maximum defect $n-k$ is certainly reached by the entire ambient space $\mathbb{V}(k,q^m)$. However, it could happen that $\mathbb{V}(k,q^m)$ is not the subspace of minimum dimension with defect $n-k$. 
In the next property, we prove that this is equivalent to the existence of a hyperplane of $\mathbb{V}(k,q^m)$ of weight $n-1$ with respect to $U$.

\begin{theorem} \label{thm:non degenerate property}
Let $U$ be an $\fq$-subspace of $\mathbb{V}(k,q^m)$,  with $\dim_{\fq} (U)=n$, such that $\langle U\rangle_{\fqm}=\mathbb{V}(k,q^m)$ and let 
\begin{equation} 
   0<  \varepsilon_U(t_1) < \varepsilon_U(t_2)< \cdots < \varepsilon_U(t_s)=n-k,
\end{equation}
the sequence of non-zero maximum defects with respect to $U$ as defined in Definition \ref{def:non-zero max defects}.
We have $t_s<k$ if and only if there exists a hyperplane $H$ of $\mathbb{V}(k,q^m)$ such that $w_U(H)=n-1$. 
\end{theorem}
\begin{proof}
If there exists a hyperplane $H$ of $\mathbb{V}(k,q^m)$ such that $w_U(H)=n-1$, then $\varepsilon_U(H)=n-k$ and hence by Remark \ref{rem: t_i minimum dimension}
\[t_s=\min\{\dim_{\fqm} (T) \,:\, \varepsilon_U(T)=n-k\}\leq k-1.\]
On the other hand, let $t_s=k-\rho$, for some $\rho >0$, and let $T$ be a subspace of $\mathbb{V}(k,q^m)$ such that $\dim_{\fqm} (T)=t_s$ and $\varepsilon_U(T)=n-k$. As a consequence, $w_U(T)=n-\rho <n$. If $\rho=1$ the assertion is proved.  If $\rho >1$, let $\overline{u}_1 \in U \setminus T$ and let  $T_1=T+\langle \overline{u}_1 \rangle_{\fqm}$. The subspace $T_1$ is a $(k-\rho+1)$-dimensional $\fqm$-subspace such that $w_U(T_1)\geq w_U(T)+1=n-\rho+1$ and hence $\varepsilon_U(T_1)\geq n-k$. Since $n-k$ is the maximum defect, 
we get 
\[\varepsilon_U(T_1)= n-k \,\,\, \mbox{and} \,\,\,\, w_U(T_1)= n-\rho+1.\]

If $\rho =2$, the assertion is proved since $T_1$ is an hyperplane.  Otherwise, proceeding inductively after $\rho-1$ steps we get a subspace $T_{\rho-1}$ of dimension $k-1$ and weight $n-1$, proving the assertion.
\end{proof}

\begin{remark}
Let $U$ be in $\mathcal{L}_q(\mathbb{V}(k,q^m))$ with $\dim_{\fq} (U)=n$, such that $\langle U\rangle_{\fqm}=\mathbb{V}(k,q^m)$. Suppose that the integer $t_s$ in Definition \ref{def:non-zero max defects} is less than $k$, and let $T$ be an $\fqm$-subspace of $\mathbb{V}(k,q^m)$ of dimension $t_s$, such that $\varepsilon_U(T)=n-k$.
Then it is not difficult to see that the action of the linear group $\mathrm{GL}(k,q^m)$ on the subspace $U$ is determined by the action of the linear group $\mathrm{GL}(t_s, q^m)$  acting on the $\fq$-subspace $T\cap U$ of $T$. Therefore, for the purposes of our study, it will  not be restrictive to assume that  $t_s=k$;  this is equivalent, by Theorem \ref{thm:non degenerate property}, to requiring that $w_U(H)<n-1$ for every  hyperplane $H$ of $\mathbb{V}(k,q^m)$.
\end{remark}

\subsection{Decomposable and evasive subspaces}

In this subsection, we introduce some definitions related to  decomposable and evasive subspaces, reformulate them in terms of defects, and discuss the consequences of the results from the previous section.

The concept of {\it evasive sets} originated in combinatorics, where a subset $S$ of a set 
$A$ is called $c$-\textbf{evasive} with respect to a family $\mathcal{F}\subseteq 2^A$ if for every $W \in \mathcal{F}$, the size of the intersection $W\cap S$ is at most $c$.
This notion has been introduced to construct Ramsey graphs in \cite{pudlak2004pseudorandom} and later used also in coding theory, see e.g. \cite{sudakov2024evasive,guruswami2013list,guruswami2016explicit}.

More recently, some attention has been paid to $\fq$-subspaces in $\Fm^k$ which are \emph{evasive subspaces}, i.e. they are $c$-evasive with respect to the family of $\Fm$-subspaces of $\Fm^k$ of fixed dimension. 
This definition extends the notion of scattered subspaces originally introduced in \cite{blokhuis2000scattered} and later extended to $h$-scattered subspaces in \cite{lunardon2017mrd,sheekeyVdV,csajbok2021generalising}; see also \cite{bartoli2021evasive,gruica2022generalised}.

\begin{definition}
    Let $\mathbb{V}(k,q^m)$ be an $\F_{q^m}$-vector space of dimension $k$ and let $U$ be an $\fq$-subspace of $\mathbb{V}(k,q^m)$ with $\dim_{\fq} (U)=n$. Let $h$ and $r$ be positive integers such that $h \in [k]$ and $h \leq r \leq km$. We say that $U$ is an $(h,r)$\textbf{-evasive} subspace if for any $h$-dimensional $\fqm$-subspace $T$ we have $\dim_{\fq}(U\cap T)\leq r$. If $r=h$, an $(h,h)$-evasive subspace is called \textbf{$h$-scattered} subspace. Furthermore, if $h=1$, then a $1$-scattered subspace will be simply called a \textbf{scattered} subspace.\\
    By \cite[Theorem 2.3]{csajbok2021generalising}, if $U$ is an $h$-scattered  $\fq$-subspace of $\mathbb{V}(k,q^m)$ with $\dim_{\fq}(U) >k$, then  $\dim_{\fq}(U) \leq \frac{km}{h+1}$. Subspaces reaching equality in the prevoious bound are called {\bf maximum $h$-scattered subspaces}.
\end{definition}

We can rephrase the definition of $(h,r)$-evasive subspace in terms of the maximum defect of $U$ with respect to $h$-dimensional subspaces.

\begin{property} \label{prop:(h,r)-evasive-defect}
Let $U$ be in $\mathcal{L}_q(\mathbb{V}(k,q^m))$ such that $\langle U\rangle_{\fqm}=\mathbb{V}(k,q^m)$ and let $h$ and $r$ be positive integers such that $h \in [k]$ and $h \leq r \leq km$. The $\fq$-subspace $U$ is $(h,r)$-evasive subspace if and only if $\varepsilon_U(h)\leq r-h$. Also, if $U$ is $(h,r)$-evasive then $U$ is $(h',r-(h-h'))$-evasive for any $1\leq h'\leq h$. 
\end{property}
\begin{proof}
    Note that the second part is a consequence of Property \ref{prop: monotonicity_maxi_defect}.
\end{proof}

The definition of {\it decomposable} $\fq$-subspace, which we will give shortly, extends similar notions already found in the literature and corresponds, in vector terms, to the definition of {\it $\fq$-linear set} admitting subspaces of complementary weight given in \cite{adriaensen2023minimum,napolitano2022linear,zullo2023multi,bartoli2024Exceptional_sequences}.

\begin{definition} \label{def:decomp}
   Let $\mathbf{k}=(k_1,\ldots,k_t)\in\mathbb N^t$ with $k=k_1+\ldots+k_t$, $k_1\leq k_2\leq \cdots \leq k_t$ and let $\mathbf{n}=(n_1,\ldots,n_t)\in\mathbb N^t$ with $n=n_1+\ldots+n_t$ and $k_i\leq n_i$ for $i\in [t]$. An $\fq$-subspace $U$ of $\mathbb{V}(k,q^m)$ of dimension $n$ is said to be {\bf decomposable} of type $(\mathbf{k}, \mathbf{n})$ if
 \begin{equation} \label{def:U_complwe}
 U=U_1\oplus \ldots \oplus U_t,
 \end{equation}
 where $U_i$ is an $\fq$-subspace of dimension $n_i$ for $i\in [t]$ and
  \begin{equation} \label{def:F_complwe}
      \mathbb{V}(k,q^m)=F_1\oplus \ldots \oplus F_t ,
  \end{equation}
  where $F_i=\langle U_i\rangle_{\fqm}$ and  $\dim_{\fqm} (F_i)=k_i $ for  $i\in [t]$. Note that $w_U(F_i)=n_i$ and $\varepsilon_U(F_i)=n_i-k_i\geq 0$ for $i\in [t]$. In addition, we refer to the $\fqm$-subspaces $F_i$ as the {\bf components} of the decomposition \eqref{def:F_complwe} of $\mathbb{V}(k,q^m)$ and to  the $\fq$-subspaces $U_i$ as the {\bf components} of the decomposition \eqref{def:U_complwe} of $U$.
\end{definition}

In \cite[Definition 2.3]{bartoli2024Exceptional_sequences} the definition of decomposable $\fq$-subspace in $\fqm^k$  is given for $\fq$-subspaces of dimension a multiple of $m$; in  \cite[Definition 3.3 and Remark 3.5]{santonastaso2024completely} the notion of decomposable is refer to a $\fqm$-linear rank metric code and the corresponding system is a decomposable $\fq$-subspace in $\fqm^k$ of type $(\mathbf{k}, \mathbf{n})$ where $\mathbf{k}=(1,1,\dots,1)$; finally  in \cite[Definition 12]{adriaensen2023minimum}, generalizing a notion given in \cite{napolitano2022linear} for projective lines and in \cite{zullo2023multi} for projective spaces of larger dimension, the authors give the definition of projective subspaces of {\it complementary weight} with respect to an $\fq$-linear set $L_U$ in a projective space $\mathrm{PG}(d,q^m)$;  reformulating such definition in purely vector terms, we obtain Definition \ref{def:decomp}.\\

The following property is easily verified.

\begin{proposition} \label{prop:decomposable}
 Let $U$ be in $\mathcal{L}_q(\mathbb{V}(k,q^m))$ with dimension $n$. Suppose that $U$ is decomposable of type $(\mathbf{k}, \mathbf{n})$, with $n=\sum_{i=1}^t n_i$, and let $F_i$ be, for $i\in[t]$, the components of the decomposition  (\ref{def:F_complwe}). Then the  following holds
 \begin{itemize}
    
     \item[1.] If $T$ is in $\mathcal{L}(\mathbb{V}(k,q^m))$ such that $T=\sum_{j=1}^{h}F_{i_j}\oplus T'$ for $h$ distinct components $F_{i_j}$ and $T'\leq_{\fqm} \sum_{l=1,\neq i_j}^{t}F_l$, then 
     \[w_U(T)=\sum_{j=1}^{h}w_U(F_{i_j})+w_U(T') \,\,\, \text{and} \,\,\, \varepsilon_U(T)=\sum_{j=1}^{h}\varepsilon_U(F_{i_j})+\varepsilon_U(T').\,\,\, \]
     \item[2.]  The minimum dimension of a subspace of $\mathbb{V}(k,q^m)$ with defect $n-k$ is $k$ (i.e. $t_s=k$) if and only if the components $F_i$ for every $i\in [t]$ are minimal with respect to its defect.
     \item[3.]  If $t_s=k$, then every subspace $T$ in $\mathcal{L}(\mathbb{V}(k,q^m))$ that is the sum of some components $F_i$  is minimal with respect to its defect.
 \end{itemize}
\end{proposition}
\begin{proof}
  {\it 1.} is obtained by noting that under the assigned assumptions it results
 \[\left(\sum_{j=1}^{h} F_{i_j}+T'\right)\cap U=\left(\sum_{j=1}^{h} U_{i_j}\right) \oplus (T'\cap U) \,\, \text{and} \,\, \left(\sum_{j=1}^{h} F_{j_i}\right)\cap T'=\{{\underline 0}\},\]
 hence inequalities (\ref{eq: subadd_defect}) and (\ref{eq: subadd_weight}) become  equalities.\\

 { {\it 2.}}
 Now, suppose that $t_s=k$ and suppose that the component $F_{\bar i}$, for some  $\bar i\in [t]$, is not minimal with respect to its defect. Then there exists $T'<_{\fqm} F_{\bar i}$ such that $\varepsilon_U(T')=\varepsilon_U(F_{\bar i})=n_{\bar i}- k_{\bar i}$. Then by the previous item the subspace $T=\sum_{i\neq \bar i}F_i \oplus T'$ is a proper subspace of $\mathbb{V}(k,q^m)$ such that $\varepsilon_U(T)=n-k$, againist the hypotesis $t_s=k$. On the other hand, assume that every component $F_i$ is minimal with respect to its defect and by way of contradiction suppose that $t_s<k$, then by Theorem \ref{thm:non degenerate property} there exists an hyperplane $H$ of $\mathbb{V}(k,q^m)$ of weight $n-1$, i.e. $U'=H\cap U$ is an hyperplane of $U$ and hence  for each  $ i\in [t]$, we have
 \[n_i-1\leq \dim_{\fq} (H\cap U_{i})=\dim_{\fq} (U'\cap U_{i})\leq n_{i} \,\, \,\,\text{and}\,\, \, k_i-1\leq \dim_{\fqm} (H\cap F_{i})\leq k_{i}.\]
Note that since $F_i=\langle U_i\rangle_{\fqm}$ for $i\in [t]$, then  
\[\dim_{\fq} (H\cap U_{i})=n_i \, \Longrightarrow \,  \dim_{\fqm} (H\cap F_{i})=k_i,\] and
\[\dim_{\fqm} (H\cap F_{i})=k_i-1 \, \Longrightarrow \,\, \dim_{\fq} (H\cap U_{i})=n_i-1.\] Now,  since $H$ is an hyperplane, there exists $\bar i\in [t]$ such that $\dim_{\fqm} (H\cap F_{\bar i})=k_{\bar i}-1$ and hence $\dim_{\fq} (H\cap U_{\bar i})=n_{\bar i}-1$, and this means that
\[\varepsilon_U(H\cap F_{\bar i})= n_{\bar i}-1-(k_{\bar i}-1)=\varepsilon_U(F_{\bar i})\]

which contradicts the minimality of $F_{\bar i}$ with respect to its defect.\\

{{\it 3.}} Finally, first of all note that since $t_s=k$ the whole space $\mathbb{V}(k,q^m)$ is the unique subspace having defect $n-k$.   Assume that $T\leq_{\fqm} \mathbb{V}(k,q^m)$ is the sum of $h\geq 1$ components, i.e. $T=\sum_{j=1}^h F_{i_j}$ where $ i_j\in [t]$ and note that $\varepsilon_U(T)=\sum_{j=1}^h(n_{i_j}-k_{i_j})$ by item {\it 1.}. 
 Let $T'$ be an $\fqm$-subspace of $T$ such that $\varepsilon_U(T')=\varepsilon_U(T)$. Note that by $(2)$ of Property \ref{property:defect} we may assume that $T'=\langle T'\cap U\rangle_{\fqm}$. Let $S=T'+(\sum_{ l\neq i_j}^{} F_l)$, then by item {\it 1.}  we get
\[\varepsilon_U(S)=\sum_{l\neq i_j}^{}\varepsilon_U(F_{l})+\varepsilon_U(T')=\sum_{l\neq i_j}^{}(n_{l}-k_l)+\varepsilon_U(T)=\sum_{i=1}^{t}(n_i-k_i)=n-k.\]

If $T'$ is a proper subspace of $T$, then $S$ is a proper subspace of  $\mathbb{V}(k,q^m)$ of defect $n-k$, a contradiction. Hence $T'=T$, i.e. $T$ is minimal with respect to its defect.

\end{proof}

Another evasive-like property, regarding decomposable subspaces, has been studied in connection with the problem of representability of $q$-matroids; see \cite{alfarano2024representability}.

We can now give the following definition, by considering the direct sum of subspaces which are scattered with respect to the hyperplanes.

\begin{definition} \label{def:kbold-scattered}
    Let $\mathbf{k}=(k_1,\ldots,k_t)\in\mathbb N^t$ and let $k=k_1+\ldots+k_t$. We say that $U$ is \textbf{$\mathbf{k}$-scattered with respect to the hyperplanes} in $\mathbb{V}(k,q^m)$ of type $\mathbf{n}=(n_1,\dots,n_t)$ if 
    \begin{itemize}
        \item[1.] $U$ is decomposable of type $(\mathbf{k}, \mathbf{n})$;
        \item[2.]  the component $U_i$ of the decomposition \eqref{def:U_complwe} is a  scattered subspace with respect to the hyperplanes in $F_i=\langle U_i\rangle_{\fqm}$ and $n_i > k_i$ for any $i\in [t]$;
        \item[3.]  $\dim_{\fq}(U\cap H)\leq k-1 \, \mbox{for each hyperplane}  \,\,   H\mbox{ such that }  F_i\not\subseteq H \mbox{ for } i\in [t],$
        or equivalently
         \[
 \varepsilon_U(H)\leq 0 \, \mbox{for each hyperplane}  \,\,   H \mbox{ such that }  F_i\not\subseteq H \mbox{ for } i\in [t].
\]
\end{itemize} 
\end{definition}

As proved in \cite[Lemma 2.3]{alfarano2024representability} Condition $3.$ of the previous definition can be extended to any subspace that does not contain any component of the decomposition. 

\begin{proposition} \label{prop:nonpositive_defect}
  Let $U$ be  a $\mathbf{k}$-scattered subspace with respect to the hyperplanes in $\mathbb{V}(k,q^m)$ with components $F_i$ for $i\in [t]$. Then 
  \[\varepsilon_U(T)\leq 0 \,  \mbox{ for each subspace }  \,\,   T \mbox{ of dimension at most k-1 such that }  F_i\not\subseteq T \mbox{ for } i\in [t].\]
\end{proposition}

It is possible to prove the following characterization of decomposable subspaces of type $(\mathbf{k}, \mathbf{n})$ that are also $\mathbf{k}$-scattered.

\begin{theorem} \label{thm:caracK-scattered}
Let $U$ be a decomposable subspace of type $(\mathbf{k}, \mathbf{n})$ in $\mathbb{V}(k,q^m)$, where $\mathbf{k}=(k_1,\ldots,k_t)$, $\mathbf{n}=(n_1,\ldots,n_t)\in\mathbb N^t$ with $k=k_1+\ldots+k_t$, $n=n_1+\ldots+n_t$ and $k_i<n_i$ for $i\in [t]$ and let $U_i$ be for $i\in [t]$ the components of its decomposition.
   The subspace $U$ is $\mathbf{k}$-scattered with respect to the hyperplanes if and only if the only subspaces of $\mathbb{V}(k,q^m)$ with positive defect with respect to $U$ and minimal with respect to its defect are the subspaces which are  sum of some components $F_i=\langle U_i\rangle_{\fqm}$ for $i \in \{1,\ldots,t\}$.
\end{theorem}
\begin{proof}
   Assume that $U$ is $\mathbf{k}$-scattered with respect to the hyperplanes. Note that from $2.$ of the Definition \ref{def:kbold-scattered} we get that each component $F_i$ of the decomposition is minimal with respect to its defect, and hence by $2.$ and $3.$ of Property \ref{prop:decomposable} we get that  $t_s=k$ and that  every subspace $T$ of $\mathbb{V}(k,q^m)$ that is the sum of some components $F_i$  is minimal with respect to its defect. Finally, from Proposition \ref{prop:nonpositive_defect} follows the assertion.
   On the other hand, assume that the only subspaces of $\mathbb{V}(k,q^m)$ with positive defect with respect to $U$ and minimal with respect to its defect are the subspaces which are  sum of some components $F_i$. This implies that any subspace with positive defect contains at least a subspace $F_i$ for some $i \in \{1,\ldots,t\}$. This implies that  $3.$ of Definition \ref{def:kbold-scattered} is verified and  we also get that  every proper subspace $L$ of $F_i$ is such that $\varepsilon_U(L)\leq 0$ and hence, by Property \ref{prop:(h,r)-evasive-defect},  $2.$ of Definition \ref{def:kbold-scattered} is satisfied. 
\end{proof}

\begin{remark}
    Although the definition \ref{def:kbold-scattered} could be generalized, we focus exclusively on this version in light of the scope of the paper.
\end{remark}

\subsection{Delsarte dual of an $\fq$-subspace}

In this section, we will revisit the notion of {\it Delsarte dual} of an $\fq$-subspace introduced in \cite{csajbok2021generalising}. In particular, in Theorem \ref{thm: dual defect}, we describe how subspaces of positive defect with respect to an $\fq$-subspace $U$ correspond to subspaces with positive defects with respect to its Delsarte dual $U^d$. In Theorem \ref{thm: dual maximum defects}, we show the exact relationship between the sequence of maximum non-zero defects with respect to a subspace $U$ and the sequence of maximum non-zero defects of its Delsarte dual $U^d$.\\

We start by recalling that every $\fq$-subspace of a $k$-dimensional $\fqm$-vector space $\mathbb{V}(k,q^m)$ that generates the entire vector space is either a subgeometry or can be obtained as a projection of a subgeometry of a larger dimension vector space. This was proved in \cite{lunardon2004translation}, where the result  is formulated in projective terms. Below is the vector-version of the same result.

  \begin{theorem}(\cite[Theorem 2]{lunardon2004translation})\label{thm:projection_subg}
    Let $U$ be in $\mathcal{L}_q(\mathbb{V}(k,q^m))$ of dimension $n>k$ such that $\langle U \rangle_{\fqm}=\mathbb{V}(k,q^m)$. Then there exist:
    \begin{itemize}
        \item[1.]  an $n$-dimensional $\fq$-subgeometry $W$ of an $n$-dimensional $\fqm$-vector space $\mathbb{V}(n,q^m)$,
        \item[2.] an $(n-k)$-dimensional $\fqm$-subspace $\Gamma$ of $\mathbb{V}(n,q^m)$ such that $\Gamma \cap W=\{ \underline 0\}$,
        \item[3.] an invertible $\fqm$-linear map 
    \[\Phi \,:\, \mathbb{V}(k,q^m) \longrightarrow \frac{\mathbb{V}(n,q^m)}{\Gamma},\]
     \end{itemize}
    
    such that $\Phi (U)=\frac{W+\Gamma}{\Gamma}$.
\end{theorem}

In the next proposition, we will give an interpretation of the weight and defect of a subspace using the quotient space model described in the previous theorem.

\begin{proposition} \label{prop:geometricdefect}
    Let $\mathbb{V}(n,q^m)$ be an $n$-dimensional $\fqm$-vector space, let $W$ be an $n$-dimensional $\fq$-subgeometry  of $\mathbb{V}(n,q^m)$  and let $\Gamma$ be an $(n-k)$-dimensional $\fqm$-subspace of $\mathbb{V}(n,q^m)$ such that  
      $\Gamma \cap W=\{ 0\}$. Let $U:=\frac{W+\Gamma}{\Gamma}$ be the $n$-dimensional  $\fq$-subspace of $\mathbb V(k,q^m)=\frac{\mathbb{V}(n,q^m)}{\Gamma}$ described in the previous theorem.
      Then for any $\fqm$-subspace $T$ of $\mathbb V(k,q^m)$ with $w_{U}(T)=h$ and $\langle T \cap U \rangle_{\fqm}=T$ the following hold:
      \begin{itemize}
          \item[\bf 1.] $T=\frac{S_h^*+\Gamma}{\Gamma}$ where $S_h$ is an $h$-dimensional $\fq$-subspace of $W$ and $S_h^*$ is the extension of $S_h$ (defined in $(3)$ of Property \ref{prop:subgeometry}). In particular, $\mathbb V(k,q^m)=\frac{W^*+\Gamma}{\Gamma}$;
          \item[\bf 2.] $\varepsilon_{U}(T)=\dim_{\fqm}(S_h^*\cap \Gamma)$;
          \item[\bf 3.]  If $T'$ is an $\fqm$-subspace of $\mathbb{V}(k,q^m)$ such that $T'\subseteq T$, $\langle T' \cap U \rangle_{\fqm}=T'$ and $T'=\frac{S_{h'}^*+\Gamma}{\Gamma}$ as in $\mathbf 1.$ (where $S_{h'}$ is an $\fq$-subspace of $W$), then $S_{h'} \subseteq S_h$.
      \end{itemize}
\end{proposition}
\begin{proof}

 Recall that $h=w_{U}(T)=\dim_{\fq} ( T \cap U)$.  Then, since $\Gamma \cap W=\{ \underline 0\}$, there exists a $h$-dimensional $\fq$-subspace $S_h$ of $W$ such that $T \cap U=\frac{S_h+\Gamma}{\Gamma}$. Hence $\langle T \cap U \rangle_{\fqm}= \frac{\langle S_h \rangle_{\fqm}+\Gamma}{\Gamma}$ and so $T=\frac{S_h^*+\Gamma}{\Gamma}$. This proves \textbf{1.}. Moreover,
 \[\varepsilon_{U}(T)=h-\dim_{\fqm} (T)=h-(\dim_{\fqm} (S_h^*) - \dim_{\fqm} (S_h^*\cap \Gamma))=\dim_{\fqm} (S_h^*\cap \Gamma), \]
 hence \textbf{2.} is proved.
 Finally, from $T'\subseteq T$, we get $\frac{S_{h'}^*+\Gamma}{\Gamma} \subseteq \frac{S_{h}^*+\Gamma}{\Gamma}$ and hence
 \[\frac{W+\Gamma}{\Gamma} \cap \frac{S_{h'}^*+\Gamma}{\Gamma} \subseteq \frac{W+\Gamma}{\Gamma}\cap \frac{S_{h}^*+\Gamma}{\Gamma},\]
 implying also the following $\frac{S_{h'}+\Gamma}{\Gamma} \subseteq \frac{S_{h}+\Gamma}{\Gamma}$. Since $W\cap \Gamma=\{ 0\}$, we get $S_{h'} \subseteq S_h$. 
\end{proof}

In \cite{csajbok2021generalising}, using the quotient model described above, a duality was introduced which allows us to associate to each $n$-dimensional $\fq$-subspace $U$ of a $k$-dimensional $\fqm$-vector space $\mathbb{V}(k,q^m)$  an $\fq$-subspace $U^d$ in an $(n-k)$-dimensional  $\fqm$-vector space $\mathbb{V}(n-k,q^m)$. When $n=m$ this construction corresponds to the one described in \cite[Theorem 4.6]{sheekeyVdV}  (see also \cite{lunardon2017mrd}).
We will call this operation \emph{Delsarte duality}. In order to define it, we need the following.

Let $\mathbb{V}(n,q^m)$ be an $n$-dimensional $\fqm$-vector space and let $W$ be an $n$-dimensional $\fq$-subgeometry  of $\mathbb{V}(n,q^m)$.   Let $\sigma': W\times W\longrightarrow \fq$ be a non-degenerate
reflexive bilinear form. There exists a unique non-degenerate reflexive bilinear form  $\sigma: \mathbb{V}(n,q^m)\times \mathbb{V}(n,q^m)\longrightarrow \fqm$ such that $\sigma (\underline{v}, \underline{u} )= \sigma' (\underline{v}, \underline{u})$ whenever $\underline{v}, \underline{u} \in W$, i.e. $\sigma$ is the $\fqm$-extension of $\sigma'$ over $\mathbb{V}(n,q^m)$.  Let $\perp$ and $\perp'$ be the orthogonal complement maps
defined by $\sigma$ and $\sigma'$ on the lattices of the
$\F_{q^m}$-subspaces of $\mathbb{V}(n,q^m)$ and the $\F_{q}$-subspaces of $W$, respectively.  It is easy to
see that the operators $\perp$, $\perp'$ and $^*$ commute, i.e. 
\begin{equation}\label{eq:*perp}
    (S^{\perp'})^*=(S^*)^{\perp}
\end{equation}
for each $\F_{q}$-subspace $S$ of
$W$ (for more details see \cite[Section 3]{csajbok2021generalising}).

Let $\Gamma$ be an $(n-k)$-dimensional $\fqm$-subspace of $\mathbb{V}(n,q^m)$ such that $\Gamma \cap W=\{ \underline 0\}$ and let $\Gamma^\perp$ be the orthogonal complement of $\Gamma$ with respect to the form $\sigma$. We have that $\Gamma^\perp$ is a $k$-dimensional $\fqm$-subspace of $\mathbb{V}(n,q^m)$.  
This means that we can project $W$ from both $\Gamma$ and $\Gamma^\perp$ using the two project mappings $p_{\Gamma}$ and $p_{\Gamma^\perp}$.

\begin{definition}
Let $U=p_{\Gamma}(W)=\frac{W+\Gamma}{\Gamma}$ be the $\fq$-subspace of $\mathbb{V}(k,q^m)=\frac{\mathbb{V}(n,q^m)}{\Gamma}$ obtained by projecting in $\mathbb{V}(k,q^m)$ the $\fq$-subgeometry  $W$. The following $\fq$-subspace of  $\mathbb{V}(n-k,q^m)=\frac{\mathbb{V}(n,q^m)}{\Gamma^\perp}$ 
\begin{equation} \label{Delsarte dual}
U^d=p_{\Gamma^\perp}(W)=\frac{W+\Gamma^\perp}{\Gamma^\perp}
\end{equation}
 is called a \textbf{Delsarte dual} of $U$ (see \cite[Definition 3.2]{csajbok2021generalising}).
\end{definition}

The following result generalizes \cite[Proposition 3.1]{csajbok2021generalising}.

\begin{proposition} \label{prop:Gamma_perp} 
Let $U=p_{\Gamma}(W)$ be the $\fq$-subspace of $\mathbb{V}(k,q^m)=\frac{\mathbb{V}(n,q^m)}{\Gamma}$ obtained by projecting in $\mathbb{V}(k,q^m)$ the $\fq$-subgeometry  $W$ where $W\cap \Gamma=\{\underline 0\}$.
We have
\begin{equation}\label{eq:dim(Ud)}
\dim_{\fq}(U^d)=n-(k-t_s),
\end{equation}
where $t_s$ is the minimum dimension of an $\F_{q^m}$-subspace of $\mathbb{V}(k,q^m)$ having defect $n-k$ with respect to $U$. Moreover, the following are equivalent:
\begin{itemize}
    \item[1.]  $\dim_{\fq} (U^d)=n$;
    \item[2.] $\Gamma^\perp \cap W=\{\underline 0\}$;
    \item[3.] $w_{U}(H)< n-1$ for each hyperplane $H$ of $\mathbb{V}(n-k,q^m)$. 
\end{itemize}
Furthermore, if any of conditions  $1.$, $2.$ or $3.$ holds, then $(U^d)^d=U$.
\end{proposition}
\begin{proof}
Since $U^d=\frac{W+\Gamma^\perp}{\Gamma^\perp}$, if $\dim_{\fq}(W\cap \Gamma^{\perp})=h$, we have
\[ \dim_{\fq}(U^d)=n-\dim_{\fq}(W\cap \Gamma^{\perp})=n-h.\]
Let $S_h=W\cap \Gamma^{\perp}$, then $S_h^* \subseteq \Gamma^{\perp}$ and hence by (\ref{eq:*perp}) we have $\Gamma \subseteq (S_h^{\perp'})^*$. This means that the subspace $T=\frac{(S_h^{\perp'})^*+\Gamma}{\Gamma}$ of $\mathbb{V}(k,q^m)$ is a $(k-h)$-dimensional $\F_{q^m}$-subspace of weight at least $n-h$ (with respect to $U$) and hence it has defect at least $n-k$. Since $n-k$ is the maximum possible defect, we get $\varepsilon_U(T)=n-k$. By Remark \ref{rem: t_i minimum dimension}
\[t_s=\min\{\dim_{\fqm}(T) \,:\, T\in\mathcal{L}(\mathbb{V}(k,q^m))\,\,\text{and}\,\, \varepsilon_U(T)=n-k\}\]
(see Definition \ref{def:non-zero max defects}) and hence   $t_s\leq k-h$. 
 On the other hand, if  $T'$ is an $\F_{q^m}$-subspace of $\mathbb{V}(k,q^m)$ of dimension $t_s$ and  with defect $n-k$,  by Proposition \ref{prop:geometricdefect} we have $T'=\frac{S_{n-k+t_s}^*+\Gamma}{\Gamma}$ and $\varepsilon_{U}(T')=n-k=\dim_{\fqm}(S_{n-k+t_s}^*\cap \Gamma)$, and so $\Gamma \subseteq S_{n-k+t_s}^*$. From which, by duality, we derive that $(S_{n-k+t_s}^{\perp'})^* \subseteq \Gamma^{\perp} $, i.e. $S_{n-k+t_s}^{\perp'} \subseteq W\cap \Gamma^{\perp}$, implying that $k-t_s \leq h$. Hence, we get $h=k-t_s$. This proves \eqref{eq:dim(Ud)}. 
 The equivalence of {\it 1.} and {\it 2.} immediately follows from the above considerations. 
 The remaining equivalences are consequences of Theorem \ref{thm:non degenerate property}.
\end{proof}

From the previous proposition and by Theorem \ref{thm:projection_subg}, we have that a subspace $U$ with the property that $w_{U}(H)< n-1$ for each hyperplane $H$ and its Delsarte dual $U^d$ have the same dimension.

\begin{remark}
    Note that the definition of  Delsarte dual of an $\fq$-subspace depends on the choice of the bilinear form $\sigma'$. However, it is possible to prove that, up to the action of the linear group $\mathrm{GL}(n-k,\fqm)$, the notion of Delsarte dual of an $\fq$-subspace can be considered independent of the choice of $\sigma'$ (see \cite[Remark 3.7]{csajbok2021generalising}).
    More precisely,  different choices for $\sigma'$ provide different Delsarte duals which are $\mathrm{GL}(n-k,\fqm)$-equivalent.
    \end{remark}

We now introduce a map between the lattices $\mathcal{L}(\mathbb{V}(k,q^m))$ and $\mathcal{L}(\mathbb{V}(n-k,q^m))$, where $\mathbb{V}(k,q^m)=\frac{\mathbb{V}(n,q^m)}{\Gamma}$ and $\mathbb{V}(n-k,q^m)=\frac{\mathbb{V}(n,q^m)}{\Gamma^\perp}$. 
This map sends $\F_{q^m}$-subspaces of $\mathbb{V}(k,q^m)$ with positive defect with respect to $U$ to proper subspaces of $\mathbb{V}(n-k,q^m)$ with positive defect with respect to $U^d$.
Moreover, we can also extend the action of this map to $\mathcal{L}_q(\mathbb{V}(k,q^m))$ and we will see that it transforms subgeometries of $\mathbb{V}(k,q^m)$ with respect to $U$ into the entire space $\mathbb{V}(n-k,q^m)$.

\begin{theorem} \label{thm: dual defect}
Let  $T$ be an $\fqm$-subspace of $\mathbb{V}(k,q^m)=\frac{\mathbb{V}(n,q^m)}{\Gamma}$ with $w_{U}(T)=h$ and such that $\langle T \cap U \rangle_{\fqm}=T$. Let $T=\frac{S_h^*+\Gamma}{\Gamma}$ where $S_h$ is an $h$-dimensional $\fq$-subspace of $W$. Let $T^d$ be the $\fqm$-subspace of $\mathbb{V}(n-k,q^m)=\frac{\mathbb{V}(n,q^m)}{\Gamma^\perp}$ defined as follows
\begin{equation} \label{eq:T^d}
T^d= \frac{(S_h^*)^\perp+\Gamma^\perp}{\Gamma^\perp}.   \end{equation} 
The following hold:
\begin{itemize}
    \item[\bf 1.] $\mathbb{V}(k,q^m)^d=\{ \underline 0\}$;
    \item[\bf 2.]  $\dim_{\fqm} (T^d)=n-k-\varepsilon_{U}(T)$;
    \item[\bf 3.] $w_{U^d}(T^d)\geq n-w_{U}(T)$;
    \item[\bf 4.] $\varepsilon_{U^d}(T^d)\geq k- \dim_{\fqm} (T)$;
    \item[\bf 5.] $(T^d)^d \subseteq T$ and $\varepsilon_{U}((T^d)^d)\geq \varepsilon_{U}(T)$;
    \item[\bf 6.] If $T'$ is an $\fqm$-subspace of $\mathbb{V}(k,q^m)$ such that
    \[T'\subseteq T \,\,\,\, \mbox{and} \,\,\,\, \langle T' \cap U \rangle_{\fqm}=T',\]
we have
\[T^d\subseteq T'^d.\]
\end{itemize}
\end{theorem}
\begin{proof}
Observe that by  {\bf 1.} of Proposition \ref{prop:geometricdefect}, $T$ can be written as $T=\frac{S_h^*+\Gamma}{\Gamma}$.
Since  $\mathbb{V}(k,q^m)=\frac{W^*+\Gamma}{\Gamma}$, we notice that {\bf 1.} is trivially obtained. Now, by (\ref{eq:*perp}) and by $\mathbf 2.$ of Proposition \ref{prop:geometricdefect} we get {\bf 2.}, indeed
\[ \dim_{\fqm} (T^d)= \dim_{\fqm} ((S_h)^{\perp'})^* - \dim_{\fqm} ((S_h^*)^\perp \cap \Gamma^\perp)= n-h -(n- \dim_{\fqm} (S_h^* + \Gamma))= \]
\[=n-h-(n-h-n+k+\dim_{\fqm}(S_h^* \cap \Gamma))=n-k-\varepsilon_{U}(T). \]
To prove {\bf 3.}, we note that since $\frac{(S_h)^{\perp'}+\Gamma^\perp}{\Gamma^\perp} \subseteq {T}^d \cap \frac{W+\Gamma^\perp}{\Gamma^\perp}$, we have $w_{U^d}(T^d)\geq n-h=n-w_{U}(T)$. 
To this point, {\bf 4.} can be easily obtained by combining {\bf 2.} and {\bf 3.}: 
\[\varepsilon_{U^d}(T^d)=w_{U^d}(T^d)-\dim_{\fqm} (T^d)\geq n-w_{U}(T) -n+k+\varepsilon_{U}(T)=k- \dim_{\fqm} (T).\]
Let us prove {\bf 5.}.  Since
 \[\frac{S_h^{\perp'}+\Gamma^\perp}{\Gamma^\perp} \subseteq {T}^d \cap \frac{W+\Gamma^\perp}{\Gamma^\perp},\]
  by $\mathbf 1.$ of Proposition \ref{prop:geometricdefect}, we can write ${T}^d=\frac{S_l^*+\Gamma^\perp}{\Gamma^\perp}$ where $S_l$ is an $\fq$-subspace of $W$ containing $S_h^{\perp'}$, so $S_l^{\perp' }\subseteq S_h$ and $(S_l^*)^{\perp}\subseteq S_h^*$. Hence
 \[ ({T}^d)^d=\frac{(S_l^*)^\perp+\Gamma}{\Gamma} \subseteq  \frac{S_h^*+\Gamma}{\Gamma}={T}.\]
 Moreover, using {\bf 4.} and {\bf 2.}, we can write
 \[\varepsilon_{{U}}(({T}^d)^d)\geq n-k-\dim_{\fqm} ({T}^d)=\varepsilon_{{U}}({T}).\]
Finally, in the hypotheses of $\mathbf 6.$, by $\mathbf{3.} $ of Proposition \ref{prop:geometricdefect} we can write $T'=\frac{S_{h'}^*+\Gamma}{\Gamma}$ where $S_{h'} \subseteq S_h$, and hence   $S_{h}^{\perp '} \subseteq S_{h'}^{\perp '}$. From which we get $(S_{h}^{\perp '})^* \subseteq (S_{h'}^{\perp '})^*$ and hence $T^d\subseteq T'^d$.
\end{proof}

Unlike a classical duality, the previous inequalities and containments are not always equalities. In the following result, we characterize the cases of equality.
More precisely, we establish under which conditions the inequalities in {\bf 3.}, {\bf 4.} and {\bf 5.} of the previous theorem become equalities.

\begin{theorem} \label{thm: =dual defect}
Let $T$ be an $\fqm$-subspace of $\mathbb{V}(k,q^m)=\frac{\mathbb{V}(n,q^m)}{\Gamma}$ with $w_{U}(T)=h$ such that $\langle T \cap U \rangle_{\fqm}=T$ and $\varepsilon_{U}(T)>0$. If $T$ is minimal with respect to its defect, then the inequalities in {\bf 3.}, {\bf 4.} and {\bf 5.} of Theorem \ref{thm: dual defect} are equalities, i.e.
\begin{itemize}  
    \item[\bf 1.] $w_{U^d}(T^d) =n-w_{U}(T)$;
    \item[\bf 2.] $\varepsilon_{U^d}(T^d)=k- \dim_{\fqm} (T)$;
    \item[\bf 3.] $(T^d)^d = T$.    
\end{itemize}
Moreover, $T^d$ is minimal with respect to its defect as well.
\end{theorem}
\begin{proof}
Following the proof of the previous theorem we have
\[T^d=\frac{S_l^*+\Gamma^\perp}{\Gamma^\perp} \,\, \mbox{and} \,\,  T^d=  \langle T^d \cap U^d \rangle_{\fqm},\] where $w_{U^d}(T^d)=l$ and  $S_l$ is an $l$-dimensional $\fq$-subspace of $W$ containing $S_h^{\perp'}$. If $S_l=S_h^{\perp'}$, then in {\bf 3.}, {\bf 4.} and {\bf 5.} of Theorem \ref{thm: dual defect} we have equalities.
By contradiction, assume that $S_h^{\perp'}\subset S_l$, that is, $l=n-h+r$ with $r>0$. 
In particular, by 2. of Theorem \ref{thm: dual defect}, $\varepsilon_{U^d}(T^d)=l-\dim_{\fqm} (T^d)= n-h+r-(n-k-\varepsilon_{U}(T))=k+r-\dim_{\fqm} (T)$. 
By applying the previous theorem to the $\fqm$-subspace ${T}^d$ of $\mathbb{V}(n-k,q^m)$ we get

\[ \dim_{\fqm} ((T^d)^d)=k-\varepsilon_{U^d}(T^d)=\dim_{\fqm} (T)-r< \dim_{\fqm} (T)  \]
and hence $(T^d)^d \subset T$; also
\[ \varepsilon_{U}((T^d)^d)= w_{U}((T^d)^d)-\dim_{\fqm} ((T^d)^d) \geq n-l-(\dim_{\fqm} (T)-r)= h- \dim_{\fqm} (T)=\varepsilon_{U}(T),\]

and this contradicts the minimality of $T$.\\
Finally, we prove the minimality of $T^d$. Let $\tilde{T}$ be a subspace of $T^d$ such that $\langle \tilde{T} \cap U \rangle_{\fqm}=\tilde{T}$,  $\varepsilon_{U^d}(\tilde{T})=\varepsilon_{U^d}(T^d)$ and minimal with respect to its defect. Then by $\mathbf 6.$ and $\mathbf 2.$ of Theorem \ref{thm: dual defect} and by $\mathbf 3.$ of this theorem, we have $T\subseteq\tilde{T}^d$ and $\dim_{\fqm}( \tilde{T}^d)=k- \varepsilon_{U^d}(\tilde{T})=\dim_{\fqm} (T)$ and hence $\tilde{T}^d=T$. Since $\tilde{T}$ is minimal with respect to its defect, by $\mathbf 3.$ we get $\tilde{T}=T^d$.
\end{proof}

The Delsarte duality clearly does not provide a one-to-one correspondence between the subspaces in 
$\mathcal{L}(\mathbb{V}(k,q^m))$ and  $\mathcal{L}(\mathbb{V}(n-k,q^m))$, but if we restrict our attention on the subspaces in  
\[ \mathcal{E}_U(\mathbb{V}(k,q^m))=\{ T \in \mathcal{L}(\mathbb{V}(k,q^m)) \colon T \text{ is a proper subspace, minimal w.r.t. its positive defect w.r.t. }U \}, \]
we can show a one-to-one correspondence between $\mathcal{E}_U(\mathbb{V}(k,q^m))$ and $\mathcal{E}_{U^d}(\mathbb{V}(n-k,q^m))$, as a consequence of the above theorems. 

\begin{corollary}
    Let $U$ be an $n$-dimensional $\fq$-subspace of a $k$-dimensional $\fqm$-vector space $\mathbb{V}(k,q^m)$ such that $w_{{U}}(H)< n-1$ for each hyperplane $H$  of $\mathbb{V}(k,q^m)$. Let $U^d$ be a Delsarte dual of $U$. 
    The map
    \[ \varphi \colon T \in \mathcal{E}_U(\mathbb{V}(k,q^m)) \mapsto T^d \in \mathcal{E}_{U^d}(\mathbb{V}(n-k,q^m)), \]
    where $T^d$ is defined in Theorem \ref{thm: dual defect}, is a one-to-one correspondence and $\varphi$ reverses the order of the set inclusion.
\end{corollary}

As a consequence of the previous results we are now able to describe the sequence of the maximum non-zero defects with respect to the Delsarte dual $U^d$ of an $\fq$-subspace $U$.  

\begin{theorem} \label{thm: dual maximum defects}
Let $U$ be an $n$-dimensional $\fq$-subspace of a $k$-dimensional $\fqm$-vector space $\mathbb{V}(k,q^m)$ such that $w_{{U}}(H)< n-1$ for each hyperplane $H$  of $\mathbb{V}(k,q^m)$. Let $U^d$ be a Delsarte dual of $U$. 
If $\varepsilon_U(t)>0$ for some $t\in \{1,\dots,k-1\}$, then 
\begin{equation} \label{eq: lower bound max defect}
\varepsilon_{U^d}(n-k-\varepsilon_U(t))\geq k-t>0.
\end{equation}

Moreover, if
\[0< \varepsilon_U(t_1)=\varepsilon_1 < \cdots < \varepsilon_U(t_{s-1})=\varepsilon_{s-1}<\varepsilon_U(k)=n-k
\]

is the sequence of maximum non-zero defects with respect to $U$, then for every $i\in \{1,\dots, s-1\}$,

\begin{equation} \label{eq: max def equal}
\varepsilon_{U^d}(n-k-\varepsilon_U(t_i)) = k-t_i.    
\end{equation}

In particular, 
\begin{equation} \label{eq: correspondence of minimum dimension subspaces}
T \leq_{\fqm} \mathbb{V}(k,q^m) \,: \, \dim_{\fqm} (T)=t_i \,\mbox{and} \,\,  \varepsilon_U(T)=\varepsilon_U(t_i)
\end{equation}
\mbox{if and only if }
\begin{equation*}
T^d \leq_{\fqm} \mathbb{V}(n-k,q^m) \,: \, \dim_{\fqm} (T^d)=n-k-\varepsilon_U(t_i) \,\mbox{and} \, \,\varepsilon_{U^d}(T^d)=\varepsilon_{U^d}(n-k-\varepsilon_U(t_i))=k-t_i.
\end{equation*}

Hence,
\begin{equation} \label{eq:sequence max def dual}
0< \varepsilon_{U^d}(n-k-\varepsilon_{s-1})=k-t_{s-1} < \varepsilon_{U^d}(n-k-\varepsilon_{s-2})=k-t_{s-2} <\cdots < \varepsilon_{U^d}(n-k-\varepsilon_1)=k-t_{1}<\varepsilon_{U^d}(n-k)=k.
\end{equation}
is the sequence of maximum non-zero defects with respect to $U^d$.
In particular, the lengths of the sequence of maximum non-zero defects with respect to $U$ and $U^d$ are equal.
\end{theorem}

\begin{proof}
  Let $t\in \{1,\dots,k-1\}$ be such that $\varepsilon_U(t)>0$ and let $T$ be a $t$-dimensional $\fqm$-subspace of $\mathbb{V}(k,q^m)$ such that  $\varepsilon_U(T)=\varepsilon_U(t)$. Then by {\it 2.} of Property \ref{property:defect} we have $\varepsilon_U(T')=\varepsilon_U(T)$ where $T'=\langle T\cap U\rangle_{\fqm}$. Hence by $\mathbf 2.$ and $\mathbf 4.$ of Theorem \ref{thm: dual defect} we get
  \[\dim_{\fqm} (T'^d)=n-k-\varepsilon_U(t) \,\, \mbox{and} \,\, \varepsilon_{U^d} (T'^d)\geq k-\dim_{\fqm} (T')\geq k-t. \]
From which we get (\ref{eq: lower bound max defect}), as
  \[\varepsilon_{U^d} (n-k-\varepsilon_U(t))\geq \varepsilon_{U^d} (T'^d)\geq k-t.\]

  Now, let $i\in \{1,\dots, s-1\}$ and let $T$ be a $t_i$-dimensional $\fqm$-subspace of $\mathbb{V}(k,q^m)$ such that  $\varepsilon_U(T)=\varepsilon_U(t_i)$. Then by Remark \ref{rem: t_i minimum dimension}, $T=\langle T\cap U\rangle_{\fqm}$ and $T$ is minimal with respect to its defect and hence by $\mathbf 2.$ of Theorem \ref{thm: dual defect} and $\mathbf 2.$ of Theorem \ref{thm: =dual defect} we get 
   \[\dim_{\fqm} (T^d)=n-k-\varepsilon_U(t_i) \,\, \mbox{and} \,\,
 \varepsilon_{U^d} (T^d)=  k-t_i.\]
Suppose that $\varepsilon_{U^d} (n-k-\varepsilon_U(t_i))=  k-t_i+\rho$ and note that, by (\ref{eq: lower bound max defect}),  $\rho \geq 0$. Let $\tilde{T}$ be an $\fqm$-subspace of $\mathbb{V}(n-k,q^m)$ such that $\dim_{\fqm} (\tilde{T})=n-k-\varepsilon_U(t_i)$ and  $\varepsilon_{U^d} (\tilde{T})=  k-t_i+\rho$. Then by Remark \ref{rem: t_i minimum dimension} and $\mathbf 2.$ of Theorem \ref{thm: dual defect} and $\mathbf 2.$ of Theorem \ref{thm: =dual defect}, we have that 
\[\dim_{\fqm} (\tilde{T}^d)=k-\varepsilon_{U^d}(\tilde{T})=t_i-\rho \,\, \mbox{and} \,\,
 \varepsilon_{U} (\tilde{T}^d)= n-k- \dim_{\fqm} (\tilde{T})=\varepsilon_U(t_i).\]
 By the minimality of $t_i$, as pointed out in Remark \ref{rem: t_i minimum dimension}, we get $\rho =0$ and hence we obtain \eqref{eq: max def equal}. 
As a consequence, we also obtain \eqref{eq: correspondence of minimum dimension subspaces}.
Finally, \eqref{eq:sequence max def dual} easily follows by combining  \eqref{eq: max def equal} and  \eqref{eq: correspondence of minimum dimension subspaces}.
  \end{proof}

\begin{remark}
    We underline here that for any $l$ such that $t_i\leq l < t_{i+1}$ for some $i \in [s]$, the value of $\varepsilon_U(l)$ remains the same of $\varepsilon_U(t_i)$. A similar situation occurs for $U^d$.
\end{remark}

\subsection{Delsarte dual of some classes of subspaces}

Our aim is now to study how the Delsarte duality behaves on an $\fq$-subspace $U$ when its sequence of maximum non-zero defects is relatively short.
In this section we will assume that $w_U(H)<\dim_{\fq}(U)-1$ for all the hyperplanes $H$ in $\mathbb{V}(k,q^m)$.
Clearly, the first case to consider is the case in which such a sequence presents only one element, that is $\varepsilon_U(k)=n-k$.
This implies that with respect to the subspace $U$ there are no elements in $\mathcal{L}(\mathbb{V}(k,q^m))\setminus \{\mathbb{V}(k,q^m)\}$ with positive defect, and so $U$ turns out to be scattered with respect to the hyperplanes.
More generally, if $\varepsilon_U(h+1)$, with $0<h<k$, is the first element in the sequence of non-zero maximum defects, then $U$ is an $h$-scattered subspace. In the following we show that the Delsarte dual of an $h$-scattered subspace is $h'$-scattered and the value of $h'$ strongly depends on the maximum of non-trivial defects appearing in the sequence of maximum non-zero defects with respect to $U$.

\begin{theorem}\label{prop:chardualhscatt}
    Let $U$ be an $n$-dimensional $\fq$-subspace of a $k$-dimensional $\fqm$-vector space $\mathbb{V}(k,q^m)$ such that $w_{{U}}(H)< n-1$ for each hyperplane $H$  of $\mathbb{V}(k,q^m)$. Let $U^d$ be a Delsarte dual of $U$.
    Assume that the subspace $U$ is $h$-scattered with  maximum non-trivial defect of $U$ equals to $\varepsilon=\varepsilon_U(t_{s-1})$. We have that $U^d$ is $(n-k-\varepsilon-1)$-scattered.
    
    In particular, if $\varepsilon<n-k-1$ then $U^d$ is scattered.
\end{theorem}
\begin{proof}
    The statement follows directly from Theorem \ref{thm: dual maximum defects}.
\end{proof}

As a consequence we derive \cite[Theorem 3.3]{csajbok2021generalising}, which holds for $h$-scattered subspaces of maximum dimension.

\begin{corollary}\label{cor:dualmaxhscatt}
    Let $U$ be a maximum $h$-scattered $\fq$-subspace of a $k$-dimensional $\fqm$-vector space $\mathbb{V}(k,q^m)$ with $m\geq h+2$.
    A Delsarte dual of $U$ is a maximum $(m-h-2)$-scattered $\fq$-subspace of a $(\frac{km}{h+1}-k)$-dimensional $\fqm$-vector space $\mathbb{V}(km/(h+1)-k,q^m)$.
\end{corollary}
\begin{proof}
    By \cite[Theorem 2.7]{csajbok2021generalising}, for any hyperplane $H$ in $\mathbb{V}(k,q^m)$ we have $\dim_{\fq}(H\cap U)\leq \frac{km}{h+1}-m+h$ which is strictly less than $km/(h+1)-1$, as $m\geq h+2$.
    Moreover, by \cite[Theorem 7.1]{zini2021scattered} there exists a hyperplane $H$ of $\mathbb{V}(k,q^m)$ such that $\dim_{\fq}(H\cap U)= \frac{km}{h+1}-m+h$, therefore we have  $\varepsilon_U(k-1)=\frac{mk}{h+1}-m+h-k+1$, i.e. $\varepsilon_U(t_{s-1})=\frac{mk}{h+1}-m+h-k+1$, and applying Theorem \ref{prop:chardualhscatt} we obtain the assertion.
\end{proof}

As already observed, the Delsarte duality preserves the length of the sequence of maximum non-zero defects.
Let us focus on those subspaces presenting only one non-trivial element in such a sequence. We will call these subspaces $1$-\textbf{defect} subspaces.

\begin{proposition}\label{prop:dual1-defect}
    Let $U$ be an $n$-dimensional $\fq$-subspace of a $k$-dimensional $\fqm$-vector space $\mathbb{V}(k,q^m)$ which is a $1$-defect subspace, i.e. its the sequence of maximum non-zero defects is
    \[ 0 < \varepsilon_U(\ell)  <\varepsilon_U(k)=n-k,\]
    for some $\ell \in \{1,\ldots,k-1\}$.
    Then the sequence of maximum non-zero defects of an its Delsarte dual $U^d$ is 
    \[0< \varepsilon_{U^d}(n-k-\varepsilon_U(\ell))=k-\ell <\varepsilon_{U^d}(n-k)=k.
    \]
\end{proposition}
\begin{proof}
    The proof follows from the definition of defect and again by Theorem \ref{thm: dual maximum defects}.
\end{proof}

The {\it club} gives a classical example of $1$-defect subspace. Let $i \in \{2,\ldots,m\}$, a subspace $U$ of $\mathbb{V}(2,q^m)$ is called an $i$-\textbf{club} if there exists a one-dimensional subspace $T \in \mathcal{L}(\mathbb{V}(2,q^m))$ such that
\[ \dim_{\fq}(U\cap T)=i, \]
and all the other one-dimensional subspaces $T'$ are such that
\[ \dim_{\fq}(U\cap T')\leq 1. \]
As a consequence of  Proposition \ref{prop:dual1-defect}, we obtain that the Delsarte dual $U^d$ of an $i$-club $U$ is a $1$-defect subspace with $|\mathcal{E}_{U^d}(\mathbb{V}(n-2,q^m))|=1$.
\begin{proposition}
    Let $U$ be an $n$-dimensional ($n> 3$) $i$-club of $\mathbb{V}(2,q^m)$. We have that its sequence of maximum non-zero defect is
    \[ 0< \varepsilon_U(1)=i-1<\varepsilon_U(2)=n-2, \]
    and $\mathcal{E}_U(\mathbb{V}(2,q^m))=\{T\}$.
    In particular, an $i$-club is a $1$-defect subspace.
    Moreover, the sequence of maximum non-zero defect of $U^d$ is 
    \[ 0< \varepsilon_{U^d}(n-1-i)=1<\varepsilon_{U^d}(n-2)=2, \]
    and $\mathcal{E}_{U^d}(\mathbb{V}(n-2,q^m))=\{T^d\}$.
\end{proposition}

\begin{example}
   Applying the previous result to a $2$-club $U$  in $\mathbb{V}(2,q^m)$ of dimension $m\geq 5$, we get that in $\mathbb{V}(m-2,q^m)$  there exists only one hyperplane with defect one with respect to $U^d$, i.e. having weight $m-2$.
    This was proven in \cite[Remark 6.14]{sheekeyVdV} for $m=5$ using MacWilliams identities, whereas our approach managed to avoid their use.
\end{example}

Finally, we show that the Delsarte dual of a decomposable subspace is still a decomposable subspace.  Then we use this result to prove that the $\mathbf{k}$-scattered property is closed under Delsarte duality.

\begin{theorem}
    Let $\mathbf{k}=(k_1,\ldots,k_t)$,  $\mathbf{n}=(n_1,\ldots,n_t)\in\mathbb N^t$ with $k=k_1+\ldots+k_t$,  $n=n_1+\ldots+n_t$  and $k_i< n_i$ for $i\in [t]$. Consider a decomposable $\fq$-subspace $U$ of $\mathbb{V}(k,q^m)$ of dimension $n$  of type $(\mathbf{k}, \mathbf{n})$.
    We have that a Delsarte dual $U^d$ of $U$ is a decomposable subspace in $\mathbb{V}(n-k,q^m)$ of type $(\mathbf{n}-\mathbf{k}, \mathbf{n})$.
\end{theorem}

\begin{proof}
   Consider the representation of $U$ in the quotient model, for which we have $\mathbb{V}(k,q^m)=\frac{\mathbb{V}(n,q^m)}{\Gamma}$ and  $U=\frac{W+\Gamma}{\Gamma}$, where $W$ is a $n$-dimensional $\fq$-subgeometry of $\mathbb{V}(n,q^m)$. Since $U$ is decomposable of type $(\mathbf{k},\mathbf{n})$, we may write
   \[U=U_1\oplus \ldots \oplus U_t=\frac{S_1+\Gamma}{\Gamma}\oplus \cdots \oplus \frac{S_t+\Gamma}{\Gamma}\]
   and 
   \[\mathbb{V}(k,q^m)=F_1\oplus  \cdots \oplus F_t=\langle U_1 \rangle_{\fqm}\oplus \cdots \oplus \langle U_t\rangle_{\fqm}=\frac{S_1^*+\Gamma}{\Gamma}\oplus \cdots \oplus \frac{S_t^*+\Gamma}{\Gamma},\]
   where $S_i$ is an $n_i$-dimensional $\fq$-subspace of $W$.
   Also, since $W\cap \Gamma=\{\underline 0\}$, we have
   \[W=S_1\oplus \cdots \oplus S_t.\]
   Then it is easy to check that if $H_i=\sum_{j=1, j\neq i}^{t}S_j$, then 
   \begin{equation} \label{eq:W}
       W=\bar{S}_1\oplus \cdots \oplus \bar{S}_t
   \end{equation}
   where $\bar{S}_i=(H_i)^{\perp'}=\cap_{j=1, j\neq i}^{t} S_j^{\perp'}$ and $\dim_{\fq}(\bar{S}_i)=n_i$. 
   Also, if $I=\{i_1,i_2,\dots, i_s\} \subseteq [t]$, since $\sum_{j=1}^s S_{i_j}=\cap_{i\in [t]\setminus I} H_i $, we have  
   \begin{equation} \label{Eq:S partial sum}
      \cap_{j=1}^{s}S_{i_j} ^{\perp'}=\sum_{i\in [t]\setminus I}^{s}\bar{S}_i.
   \end{equation}
   Let $F'_i=\sum_{j=1, j\neq i}^{t}F_j$ for $i\in [t]$ and note that 
   \[F'_i=\sum_{j=1, j\neq i}^{t} \frac{S_j^*+\Gamma}{\Gamma}=\frac{(\sum_{j=1, j\neq i}^{t}S_j)^*+\Gamma}{\Gamma}=\frac{H_i^*+\Gamma}{\Gamma}.\]
   Clearly, the subspaces $F'_i$ have positive defect and, by {\it 3.} of Proposition \ref{prop:decomposable}, they are  minimal with respect to their defects. Hence, by Theorems \ref{thm: dual defect} and  \ref{thm: =dual defect} we have that 
   \[(F_i')^d=\frac{(H_i^{\perp'})^*+\Gamma^{\perp}}{\Gamma^{\perp}}=\frac{\bar{S}_i^*+\Gamma^{\perp}}{\Gamma^{\perp}},\]
   \[\dim_{\fqm}((F_i')^d)=n-k-\varepsilon_U(F_i')=n_i-k_i,\] 
   \[ w_{U^d}((F_i')^d)=n-w_U(F_i')=n_i,\] 
   \[\varepsilon_{U^d}((F_i')^d)=k-\dim_{\fqm}(F_i')=k_i,\] 
   and
   \[((F_i')^d)^d=F_i'.\]
   Also, $(F_i')^d$ is minimal with respect to its defect. By Equation (\ref{eq:W}) we get
   
   \begin{equation} \label{eq:decomposition_Ud}
       U^d=\frac{W+\Gamma^{\perp}}{\Gamma^{\perp}}=\frac{\bar{S}_1+\Gamma^{\perp}}{\Gamma^{\perp}}\oplus\cdots\oplus \frac{\bar{S}_t+\Gamma^{\perp}}{\Gamma^{\perp}}
   \end{equation}
   and hence, since $(U^d)^*=\mathbb{V}(n-k,q^m)$, we have
   \[\mathbb{V}(n-k,q^m)=(F_1')^d+\cdots +(F_t')^d.\]
   
and now comparing the dimensions we get
\begin{equation} \label{eq:decomposition_Vd}
\mathbb{V}(n-k,q^m)=(F_1')^d\oplus\cdots \oplus (F_t')^d,
\end{equation}
i.e. $U^d$ is decomposable of type $(\mathbf{n}-\mathbf{k}, \mathbf{n})$.
\end{proof}

\begin{remark}
     In \cite{bartoli2024Exceptional_sequences}, the authors proved that the Delsarte dual of decomposable subspaces in $\fqm^k$ having dimension a multiple of $m$ is again a decomposable subspace. Therefore, the previous result can be seen as an extension of \cite[Proposition 2.8]{bartoli2024Exceptional_sequences}.
\end{remark}

Now, we are able to prove that the property of being $\mathbf{k}$-scattered with respect to the hyperplanes is preserved under the Delsarte duality.

\begin{theorem} \label{th:dual k-bold scattered}
 Let $\mathbf{k}=(k_1,\ldots,k_t)\in\mathbb N^t$ with $k=k_1+\ldots+k_t$ and let $U$ be an $\fq$-subspace  $\mathbf{k}$-scattered with respect to the hyperplanes in $\mathbb{V}(k,q^m)$ of type $\mathbf{n}=(n_1,\dots,n_t)$.
 Then the Delsarte dual $U^d$ of $U$ is $(\mathbf{n}-\mathbf{k})$-scattered with respect to the hyperplanes in $\mathbb{V}(n-k,q^m)$ of type $\mathbf{n}$.    
\end{theorem}
\begin{proof}
From the previous theorem $U^d$  is decomposable of type $(\mathbf{n}-\mathbf{k}, \mathbf{n})$ in $\mathbb{V}(n-k,q^m)$ with decompositions (\ref{eq:decomposition_Ud}) and (\ref{eq:decomposition_Vd}).
Now, by Theorem \ref{thm: =dual defect} and by Theorem \ref{thm:caracK-scattered}, a subspace $T$ of $\mathbb{V}(n-k,q^m)$ with positive defect with respect to $U^d$ and minimal with respect to its defect is the dual of a subspace of $\mathbb{V}(k,q^m)$ which is  sum of some subspaces $F_i$, i.e. there exists $I=\{i_1,\ldots,i_h\}\subseteq\{1,\ldots,t\}$ such that
\[T=\left(\sum_{j=1}^{h} F_{i_j}\right)^d=\left(\frac{\sum_{j=1}^{h}S_{i_j}^*+\Gamma}{\Gamma}\right)^d=\frac{(\cap_{j=1}^{h}S_{i_j}^{\perp'})^*+\Gamma}{\Gamma}\]
and by (\ref{Eq:S partial sum}) we can write
\[T=\frac{(\sum_{i\in [t]\setminus I }^{}\bar{S}_{i}^*)+\Gamma}{\Gamma}=\sum_{i\in [t]\setminus I}^{} F_i',\]
i.e. $T$ is the sum of some components $F'_i$, where $\bar{S}_i=\cap_{j=1, j\neq i}^{t} S_j^{\perp'}$ for any $i$.
Hence by Theorem \ref{thm:caracK-scattered}, we obtain that $U^d$ is $(\mathbf{n}-\mathbf{k})$-scattered with respect to the hyperplanes in $\mathbb{V}(n-k,q^m)$ of type $(n_1,\dots,n_t)$. 
\end{proof}

\section{Rank-metric codes and $q$-systems: a new framework}

In this section, after recalling some basics of rank-metric codes, we will propose a new framework for the geometric correspondence between rank-metric codes and $q$-systems. This will help us to use the results of the previous section to prove, as a byproduct, some previously known results and to detect classes of codes that are closed under duality.

\subsection{Preliminaries on rank-metric codes and systems}

Rank metric codes were introduced by Delsarte \cite{de78} in 1978 as subsets of matrices and they have been intensively investigated in recent years because of their applications; we refer to \cite{sheekeysurvey,polverino2020connections}.
In this section we will be interested in rank metric codes in $\F_{q^m}^n$.
The \textbf{rank} (weight) $w(v)$ of a vector $v=(v_1,\ldots,v_n) \in \F_{q^m}^n$ is defined as $w(v)=\dim_{\fq} (\langle v_1,\ldots, v_n\rangle_{\fq})$. 

A \textbf{(linear vector) rank metric code} $\C $ is an $\F_{q^m}$-subspace of $\F_{q^m}^n$ endowed with the rank distance, where such a distance is defined as $d(x,y)=w(x-y)$, where $x, y \in \F_{q^m}^n$. 
For details and applications we refer to \cite{bartz2022rank,gorla2018codes}.

Let $\C \subseteq \F_{q^m}^n$ be a linear rank metric code. We will write that $\C$ is an $[n,k,d]_{q^m/q}$ code (or $[n,k]_{q^m/q}$ code) if $k=\dim_{\F_{q^m}}(\C)$ and $d$ is its minimum distance, that is $d=\min\{d(x,y) \colon x, y \in \C, x \neq y  \}$.
Moreover, for any $i \in \{1,\ldots,n\}$, we define $A_i(\C)$ the number of codewords in $\C$ having weight $i$.

By the classification of $\F_{q^m}$-linear isometry of $\F_{q^m}^n$ (see \cite{berger2003isometries}), we say that two rank metric codes $\C,\C' \subseteq \F_{q^m}^n$ are \textbf{(linearly) equivalent} if and only if there exists a matrix $A \in \mathrm{GL}(n,q)$ such that
$\C'=\C A=\{vA : v \in \C\}$. 
The codes we will consider are \textbf{non-degenerate}, i.e. those for which the columns of any generator matrix of $\C$ are $\fq$-linearly independent. 

For rank-metric codes a Singleton-like bound holds.

\begin{theorem}(see \cite{de78}) \label{th:singletonrank}
    Let $\C \subseteq \F_{q^m}^n$ be an $[n,k,d]_{q^m/q}$ code.
Then 
\begin{equation}\label{eq:boundgen}
d \leq \min\{n-k+1,m-km/n+1\}.
\end{equation}
\end{theorem}

An $[n,k,d]_{q^m/q}$ code is called \textbf{Maximum Rank Distance code} (or shortly \textbf{MRD code}) if its parameters attain the equality in the Singleton-like bound \eqref{eq:boundgen}.
When $n>m$ an $n$ does not divide $mk$, \eqref{eq:boundgen} does not provide an integer, and therefore an MRD code cannot exist. Hence, when $n>m$ and $mk$ is not a multiple of $n$, in \cite{marino2023evasive} it has been defined an $\fqm$-linear \textbf{quasi-MRD} as an $[n,k]_{q^m/q}$ code with minimum distance $d =m- \lceil km/n \rceil +1$ (see \cite{de2018weight} for rank metric codes in $\F_q^{n\times m}$ with $n\leq m$). We say that $\C$ is \textbf{dually quasi-MRD} if and only if $\C$ and $\C^\perp$ are quasi-MRD codes.

The geometric counterpart of a rank-metric code is the $q$-system. 
An $[n,k,d]_{q^m/q}$ \textbf{system} $U$ is an $\F_q$-subspace of $\F_{q^m}^k$ of dimension $n$, such that
$ \langle U \rangle_{\F_{q^m}}=\F_{q^m}^k$ and
$$ d=n-\max\left\{\dim_{\F_q}(U\cap H) \mid H \textnormal{ is an }\F_{q^m}\textnormal{-hyperplane of } \F_{q^m}^k\right\}.$$
Moreover, two $[n,k,d]_{q^m/q}$ systems $U$ and $U'$ are \textbf{equivalent} if there exists an $\F_{q^m}$-isomorphism $\varphi\in\mathrm{GL}(k,q^m)$ such that
$$ \varphi(U) = U'.$$

Thanks to the following result, we can construct a system starting from a non-degenerate rank-metric code and conversely.

\begin{theorem}(\cite{Randrianarisoa2020ageometric,alfarano2021linearcutting}) \label{th:connection}
Let $\C$ be a non-degenerate $[n,k,d]_{q^m/q}$ rank-metric code and let $G$ be a generator matrix.
Let $U \subseteq \F_{q^m}^k$ be the $\F_q$-span of the columns of $G$.
The rank weight of an element $x G \in \C$, with $x \in \F_{q^m}^k$ is
\begin{equation}\label{eq:relweight}
w(x G) = n - \dim_{\fq}(U \cap x^{\perp}),\end{equation}
where $x^{\perp}=\{y \in \F_{q^m}^k \colon <x, y>=0\}.$ In particular,
\begin{equation} \label{eq:distancedesign}
d=n - \max\left\{ \dim_{\fq}(U \cap H)  \colon H\mbox{ is an } \F_{q^m}\mbox{-hyperplane of }\F_{q^m}^k  \right\}.
\end{equation}
\end{theorem}

The above result allows us to give a one-to-one correspondence between equivalence classes of non-degenerate $[n,k,d]_{q^m/q}$ codes and equivalence classes of $[n,k,d]_{q^m/q}$ systems, see \cite{Randrianarisoa2020ageometric}.
The system $U$ and the code $\C$ as in Theorem \ref{th:connection} are said to be \textbf{associated} and a system associated with a code $\C$ will usually be written as $U_{\C}$.

An essential concept is the rank support of a codeword.
Let $\Gamma=(\gamma_1,\ldots,\gamma_m)$ be an ordered $\fq$-basis of $\F_{q^m}$. For any vector $x=(x_1, \ldots ,x_n) \in \F_{q^m}^n$ define the matrix $\Gamma(x)\in \F_{q}^{n \times m}$, where
$$x_{i} = \sum_{j=1}^m \Gamma (x)_{ij}\gamma_j, \qquad \mbox{ for all } i \in \{1,\ldots,n\},$$
that is $\Gamma(x)$ is the matrix expansion of the vector $x$ with respect to the $\fq$-basis $\Gamma$ of $\F_{q^m}$ and this clearly preserves its rank, i.e. $w(x)=\mathrm{rk}(\Gamma(x))$.
The \textbf{rank support} of $x$ is defined as the column span of $\Gamma(x)$:
$$\mathrm{supp}(x)=\mathrm{colsp}(\Gamma(x)) \subseteq \fq^n.$$

As shown in \cite[Proposition 2.1]{alfarano2021linearcutting}, the support of a vector is independent of $\Gamma$, allowing us to talk about the support of a word without reference to $\Gamma$.
Let $\C$ be an $[n,k]_{q^m/q}$ code, we define the \textbf{support} of $\C$ as
\[ \mathrm{supp}(\C)=\sum_{x \in \C} \mathrm{supp}(x). \]
The support of the code $\C$ can be computed directly on $\C$, without using a fixed order basis.

\begin{theorem}\label{thm:traceandsupport}(\cite[Theorem 16]{jurrius2017defining})
Let $\C$ be a $[n,k,d]_{q^m/q}$ rank-metric code, we have
\[ \mathrm{supp}(\C)=\mathrm{Tr}_{q^m/q}(\C), \]
where
\[ \mathrm{Tr}_{q^m/q}(\C)=\{ (\mathrm{Tr}_{q^m/q}(c_1),\ldots,\mathrm{Tr}_{q^m/q}(c_n)) \colon (c_1,\ldots,c_n) \in \C \}. \]
\end{theorem}

For more details, refer to \cite{martinez2017relative}.

Generalized rank weights have been defined in various ways, as seen in \cite{jurrius2017defining}. This paper focuses on the definition provided in \cite{Randrianarisoa2020ageometric}, particularly the equivalent version found in \cite[Theorem 3.14]{alfarano2021linearcutting}, which is closely related to systems.
Let $\C$ be a non-degenerate $[n,k,d]_{q^m/q}$ code and let $U_{\C}$ be an associated system.
For any $r \in \{1,\ldots,k\}$, the \textbf{$r$-th generalized rank weight} is
\begin{equation}\label{eq:defgenrankweight}
d_r(\C)=n - \max\left\{ \dim_{\fq}(U_{\C} \cap T)  \colon T\mbox{ is an } \F_{q^m}\mbox{-subspace of codim. $r$ of  }\F_{q^m}^k  \right\}.
\end{equation}

Note that when $r=1$, in the above definition we obtain the minimum distance.

Jurrius and Pellikaan in \cite[Section 5]{jurrius2017defining} provided an alternative definition of the $r$-th generalized
rank weight. Combining such a result with Theorem \ref{thm:traceandsupport} we get the following.

\begin{theorem} \label{thm:generalizedweightsminsupports}(\cite[Corollary 17,Theorem 21]{jurrius2017defining})
     Let $\mathcal{C}$ be an $[n,k,d]_{q^m/q}$ rank-metric code and let $1\leq r\leq k$. Then 
        \begin{align*}
            d_r(\mathcal{C})&=\min\left\{\mathrm{dim}_{\fq}(\mathrm{supp}(\mathcal{D}))\colon\mathcal{D}\text{ is a linear subcode of }\mathcal{C}\text{ of dimension }r\right\}\\
            &=\min\left\{\mathrm{dim}_{\fq}(\mathrm{Tr}_{q^m/q}(\mathcal{D}))\colon\mathcal{D}\text{ is a linear subcode of }\mathcal{C}\text{ of dimension }r\right\}.
        \end{align*}
\end{theorem}

An important property of generalized rank weights is established by  the {\it Wei-type duality theorem} for $\fqm$-linear rank metric codes,  originally proven by Ducoat in \cite[Theorem I.3]{ducoat2015generalized} and later derived alternatively in \cite[Theorem 7]{britz2020}. This theorem reveals a fundamental connection between the generalized rank weights of a code $\mathcal{C}$ and those of its dual  $\mathcal{C}^{\perp}$. 

\begin{theorem}{\textnormal{{(Wei-type duality, \protect{\cite[Theorem I.3]{ducoat2015generalized}).}}}}
    \label{prop: Wei-type duality}
    Let $\mathcal{C}$ be an $[n,k,d]_{q^m/q}$ code. Then
            \[
            \left\{d_r(\mathcal{C}^\perp)\colon 1 \leq r \leq n-k\right\}=\{1, \ldots, n\}\setminus\left\{n+1-d_r(\mathcal{C})\colon 1 \leq r \leq k\right\}.
            \]    
\end{theorem}

Finally, we recall that for generalized weights, a Singleton-like bound holds.

\begin{theorem}\cite{martinez2016similarities}
     Let $\mathcal{C}$ be an $[n,k,d]_{q^m/q}$ code and let $1\leq r \leq k$. Then
     \[d_r(\C) \leq \min \left\{n-k+r,\frac{m}n(n-k)+m(r-1)+1\right\}. \]
\end{theorem}

Note that when $r=1$, this bound corresponds to Theorem \ref{th:singletonrank} and when $n\leq m$ the bound simply reads as 
\[d_r(\C) \leq n-k+r. \]

\subsection{A new framework}

In this section, our aim is to describe systems associated with a rank-metric code in a quotient space, similarly to what we did in the previous section. The main advantage of doing so is to achieve a completely geometric description of the Delsarte duality, where both a system associated with the code and a system associated with its dual can be observed within the same framework. Therefore, in this section, we will only consider codes that are non-degenerate and for which their duals are still non-degenerate.

\begin{theorem} \label{thm:isomorphism_systems}
    Let $\C$ be a non-degenerate $[n,k]_{q^m/q}$ code such that $\C^{\perp}$ is non-degenerate.
    Let $G$ be any generator matrix of $\C$ and $U_{\C}$ the $\fq$-span of the columns of $G$.
    The map
    \[\phi \colon x +\C^\perp \in \frac{\F_{q^m}^n}{\C^\perp} \mapsto xG^\top \in \F_{q^m}^k \]
    is an $\F_{q^m}$-linear isomorphism which maps the $\fq$-subspace $(\F_q^n+\C^\perp)/\C^\perp$ into $U_{\C}$.
\end{theorem}
\begin{proof} Note that the map
\[\varphi \,\colon\, x\in \F_{q^m}^n \mapsto xG^\top \in \F_{q^m}^k\]
     is   $\Fm$-linear and $ker \varphi = \C^\perp$. Hence $\varphi$ is surjective. This implies that $\phi$ is an $\F_{q^m}$-linear isomorphism. 
   Also 
    \[ \phi\left(\frac{\F_q^n+\C^\perp}{\C^\perp}\right)=\{xG^T\,\colon\, x\in \F_q^n\}= U_{\mathcal{C}}. \]
   
\end{proof}

By the above result, we can identify $U_{\C}$ with the $\fq$-subspace $(\F_q^n+\C^\perp)/\C^\perp$ of the quotient space $\frac{\F_{q^m}^n}{\C^\perp}$. Therefore, as an immediate consequence, we have that the duality on rank-metric codes corresponds to the Delsarte duality on their systems. This has only been proved for the case $n=m$ in \cite{lunardon2017mrd,sheekeyVdV}.

\begin{corollary}\label{cor:delsartedualitysubspaceandcode}
    Let $\C$ be a non-degenerate $[n,k]_{q^m/q}$ code such that $\C^{\perp}$ is non-degenerate. We have that
    \[ U_{\C}\simeq (\F_q^n+\C^\perp)/\C^\perp \,\,\text{and}\,\, U_{\C^\perp}\simeq (\F_q^n+\C)/\C.  \]
    In particular, $U_{\C^\perp}$ is a Delsarte dual of $U_{\C}$.
\end{corollary}

The main advantage is that we can embed both $U_{\C}$ and $U_{\C^\perp}$ in quotient spaces of the same vector space $\F_{q^m}^n$, where we can study directly the metric properties of $\C$ and $\C^\perp$ as shown in the next proposition.

\begin{proposition}
     Let $\C$ be a non-degenerate $[n,k]_{q^m/q}$ code such that $\C^{\perp}$ is non-degenerate.
     For any non-zero codeword $c \in \C$ we have
     \[ w_{U_{\C}}((\langle c\rangle_{\fqm}^\perp+\C^\perp)/\C^\perp )=\dim_{\fq}(\fq^n\cap \langle c \rangle_{\F_{q^m}}^\perp)\,\,\text{and}\,\, w(c)=n-w_{U_{\C}}(\langle c\rangle_{\fqm}^\perp). \]
\end{proposition}
\begin{proof}
    Let $c=(c_1,\ldots,c_n) \in \C\setminus \{0\}$.
    The map $\varphi_{c}\colon x \in \fq^n\mapsto \langle x,c\rangle \in \fqm$ is $\fq$-linear and its image is $\langle c_1,\ldots,c_n\rangle_{\fq}$.
    Hence, from this we derive that 
    \begin{equation}\label{eq:cond1}
    w(c)=\dim_{\fq}(\textrm{Im}(\varphi_c))=n-\dim_{\fq}(\fq^n\cap \langle c \rangle_{\F_{q^m}}^\perp).
    \end{equation}
    Now, observe that $(\langle c\rangle_{\fqm}^\perp+\C^\perp)/\C^\perp$ is an hyperplane in $\fqm^n/\C^\perp$ and
    \begin{equation} \label{eq:cond2} 
    w_{U_{\C}}\left( (\langle c\rangle_{\fqm}^\perp+\C^\perp)/\C^\perp\right)=\dim_{\F_q}\left( ((\langle c\rangle_{\fqm}^\perp+\C^\perp)/\C^\perp) \cap ((\fq^n+\C^\perp)/\C^{\perp})\right)=
    \end{equation}
\[=\dim_{\fq}\left( \frac{(\langle c\rangle_{\fqm}^\perp \cap \fq^n)+\C^\perp}{\C^\perp} \right)=\dim_{\fq}(\langle c\rangle_{\fqm}^\perp \cap \fq^n), \]
since $\C^\perp$ is non-degenerate. From Equations \eqref{eq:cond1} and \eqref{eq:cond2} we obtain the assertion. 
\end{proof}

In this framework we can also read the supports of the codewords and of the subcodes, relying on the following result by Delsarte.

\begin{theorem}(\cite[Theorem 2]{delsarte1975subfield})\label{thm:Delsartetrace}
    Let $\mathcal{D}$ be an $[n,k]_{q^m/q}$ code. We have
    \[ (\mathcal{D}^\perp\cap \fq^n)^{\perp'}=\mathrm{Tr}_{q^m/q}(\mathcal{D}). \]
\end{theorem}

The above result allows us to prove the following correspondence.

\begin{proposition}\label{prop:suppgeometric}
    Let $\C$ be a non-degenerate $[n,k]_{q^m/q}$ code such that $\C^{\perp}$ is non-degenerate.
     For any non-zero codeword $c \in\C$ we have
     \[ \mathrm{supp}(c)=\left( \langle c\rangle_{\fqm}^\perp\cap \fq^n \right)^{\perp'}, \]
     and for any subcode $\mathcal{D}$ of $\mathcal{C}$
     \[ \mathrm{supp}(\mathcal{D})=\left( \mathcal{D}^\perp\cap \fq^n \right)^{\perp'}. \]
\end{proposition}
\begin{proof}
    Clearly, $\mathrm{supp}(c)=\mathrm{supp}(\langle c\rangle_{\F_{q^m}})$. Therefore, by Theorems \ref{thm:traceandsupport} and \ref{thm:Delsartetrace} we have that
    \[ \mathrm{supp}(c)=\mathrm{Tr}_{q^m/q}(\langle c\rangle_{\F_{q^m}})=(\langle c\rangle_{\F_{q^m}}^\perp\cap \fq^n)^{\perp'}. \]
   
    A similar argument can be performed on $\mathcal{D}$ to get the second part of the assertion.
\end{proof}

We can now characterize the generalized weights in terms of the intersection with the system.

\begin{proposition} \label{prop:gener_weight_maximum_defect}
    Let $\C$ be a non-degenerate $[n,k]_{q^m/q}$ code such that $\C^{\perp}$ is non-degenerate. 
    For any $r \in [k]$, we have
    \[ d_{r}(\C)=n-k+r-\varepsilon_{U_{\C}}(k-r).\]
\end{proposition}
\begin{proof}
    By (\ref{eq:defgenrankweight}),
    \[d_r(\C)=n - \max\left\{ \dim_{\fq}(U_{\C} \cap T)  \colon T\mbox{ is an } \F_{q^m}\mbox{-subspace of codim. $r$ of  }\F_{q^m}^k  \right\}=\]
   \[=n - \max\left\{ \dim_{\F_{q^m}}(T)+\varepsilon_{U_{\C}}(T) \colon T\mbox{ is an } \F_{q^m}\mbox{-subspace of codim. $r$ of  }\F_{q^m}^k  \right\}= \]
    \[=n -k+r- \varepsilon_{U_{\C}}(k-r) .\]
    
\end{proof}

By the previous Proposition and by Proposition \ref{prop: monotonicity_maxi_defect}, we can give an alternative proof of the monotonicity property of the generalized rank weights (see e.g. \cite[Theorem I.2]{ducoat2015generalized}). 

\begin{proposition}
    Let $\C$ be a non-degenerate $[n,k]_{q^m/q}$ code such that $\C^{\perp}$ is non-degenerate. 
    The generalized weights are in increasing order.
\end{proposition}
\begin{proof}
By the above proposition, we have that for any $r \in [k]$ 
\[ d_{r}(\C)=n-k+r-\varepsilon_{U_{\C}}(k-r).\]
Now by Property \ref{prop: monotonicity_maxi_defect} we have that 
\[d_r(\C)=n-k+r-\varepsilon_{U_{\C}}(k-r) < n-k+r+1-\varepsilon_{U_{\C}}(k-r-1)=d_{r+1}(\C).\]
\end{proof}

In the next theorem, we will see that having complete knowledge of the sequence of maximum non-zero defects of the system $U_{\C}$ associated with a code $\C$ allows us to determine not only the generalized weights of $\C$, but also the generalized weights of the dual $\C^{\perp}$. As a by product, we obtain an alternative proof of the Wei-type duality theorem in the case of non-degenerate $\fqm$-linear rank-metric codes.

\begin{theorem} \label{thm:generalizedweightduality}
    Let $\C$ be a non-degenerate $[n,k]_{q^m/q}$ code such that $\C^{\perp}$ is non-degenerate. 
    Let $U_{\C}$ be a system associated with $\C$ and suppose that
    \[0=\varepsilon_{U_{\C}}(0)=\varepsilon_0< \varepsilon_{U_{\C}}(t_1)=\varepsilon_1 < \cdots < \varepsilon_{U_{\C}}(t_{s-1})=\varepsilon_{s-1}<\varepsilon_{U_{\C}}(k)=n-k
    \]
    is the sequence of maximum non-zero defects with respect to $U_{\C}$. For every $r\in \{1,\ldots,k\}$, if $t_i\leq k-r <t_{i+1}$, the $r$-th generalized weight of $\C$ is
   \begin{equation} \label{eq:gen_weight_maximum_defect}
       d_r(\C)=n-(k-r)-\varepsilon_{U_{\C}}(t_i)=n-(k-r)-\varepsilon_{i}. \end{equation} 
    In particular,
    \begin{equation} \label{eq: I}
        \{n+1-d_r(\mathcal{C})\colon 1 \leq r \leq k\}=\bigcup_{i=0}^{s-1} I_i
    \end{equation}
    where 
    $I_i=[t_{i+1}-t_i]+(\varepsilon_i+t_i)=\{1+\varepsilon_i+t_i,2+\varepsilon_i+t_i,\dots, t_{i+1}-t_i+\varepsilon_i+t_i\}$ for any $ i\in  \{0,1,\dots,s-1\}$.
    
    Also, for every $r \in \{1,\ldots,n-k\}$, let $j \in [s]$ be such that $\varepsilon_{j-1}< r\leq \varepsilon_{j}$, then the $r$-th generalized weight of $\C^\perp$ is

    \begin{equation} \label{eq:dualgeneralized weight}       
d_r(\C^\perp)=r+t_{j} 
     \end{equation}
and hence
\begin{equation} \label{eq: J}
 \{d_r(\mathcal{C}^\perp)\colon 1 \leq r \leq n-k\}=\bigcup_{j=1}^{s} J_j
\end{equation}

 where 
    $J_j=[\varepsilon_{j}-\varepsilon_{j-1}]+(\varepsilon_{j-1}+t_j)$ for $j\in [s]$.
    Finally, from Equations (\ref{eq: I}) and (\ref{eq: J}) we get
\[
            \left\{d_r(\mathcal{C}^\perp)\colon 1 \leq r \leq n-k\right\}=\{1, \ldots, n\}\setminus\left\{n+1-d_r(\mathcal{C})\colon 1 \leq r \leq k\right\}.
            \]    
    
\end{theorem}
\begin{proof}
    Recall that, by Definitions \ref{def:maximum_defect} and 
    \ref{def:non-zero max defects},   $\varepsilon_{U}(0)=\varepsilon_0=0$ and $t_0=0$ for any $\fq$-subspace $U$. Now, Equation (\ref{eq:gen_weight_maximum_defect}) follows by Proposition \ref{prop:gener_weight_maximum_defect} and by  Equation (\ref{eq:non-zero_defects_Mimim_dimensions}) of Definiton \ref{def:non-zero max defects}. By  Equation (\ref{eq:gen_weight_maximum_defect}) we get (\ref{eq: I}). Indeed, recall that for any $r$ such that $t_i\leq k-r <t_{i+1}$ we have that $n+1-d_r(\mathcal{C})=k-r+1+\varepsilon_{i}$ and so
    \[\bigcup_{r=1}^k\{n+1-d_r(\mathcal{C})\}=\bigcup_{r=1}^{k}\{k-r+1+\varepsilon_{i}\,\colon i \text{ is s.t. } \, t_i\leq k-r <t_{i+1}  \}=\bigcup_{i=0}^{s-1} ([t_{i+1}-t_i]+(\varepsilon_i+t_i)).\]
    Note that if $I_i=[t_{i+1}-t_i]+(\varepsilon_i+t_i) $, then $I_i\subseteq [n]$ and  $I_i\cap I_j=\emptyset$ if $i\neq j$. 

    By Proposition \ref{prop:gener_weight_maximum_defect} applied to $\C^\perp$
    \[d_{r}(\C^\perp)=k+r-\varepsilon_{U_{\C^\perp}}(n-k-r).\]
Now, by Corollary \ref{cor:delsartedualitysubspaceandcode}, Theorem \ref{thm: dual maximum defects} and  (\ref{eq:non-zero_defects_Mimim_dimensions}) of Definiton \ref{def:non-zero max defects}, 
   we get (\ref{eq:dualgeneralized weight}), that is
   \[d_{r}(\C^\perp)=k+r-\varepsilon_{U_{\C^\perp}}(n-k-r)=k+r-\varepsilon_{U_{\C}^d}(n-k-r)=r+t_j\]
   if $n-k-\varepsilon_j \leq n-k-r < n-k-\varepsilon_{j-1}$, i.e. $\varepsilon_{j-1}< r\leq \varepsilon_{j}$. Hence, we can write 
   \[
            \bigcup_{r=1}^{n-k}\left\{d_r(\mathcal{C}^\perp)\right\}=\bigcup_{r=1}^{n-k}\{r+t_j\,\colon j \text{ is s.t. }\, \varepsilon_{j-1}< r\leq \varepsilon_{j}  \}=\bigcup_{j=1}^{s} ([\varepsilon_{j}-\varepsilon_{j-1}]+(\varepsilon_{j-1}+t_j)),\]
            obtaining (\ref{eq: J}).
            Also in this case if $J_j=[\varepsilon_{j}-\varepsilon_{j-1}]+(\varepsilon_{j-1}+t_j)$, then $J_j \subseteq [n]$ and $J_i\cap J_j=\emptyset$ if $i\neq j$.

            Finally, we have
    \[\min (J_j)=1+\varepsilon_{j-1}+t_j=1+\max (I_{j-1})\]
    \[\max (J_j)=\varepsilon_{1}+t_j=\min (I_{j})-1,\]
    and hence $I_i\cap J_j=\emptyset$ for each $i\in \{0,\dots,s-1\}$ and $j\in [s]$. This implies     
    \[
            \left\{d_r(\mathcal{C}^\perp)\colon 1 \leq r \leq n-k\right\}=\{1, \ldots, n\}\setminus\left\{n+1-d_r(\mathcal{C})\colon 1 \leq r \leq k\right\}.
            \] 
\end{proof}

\subsection{Families of codes closed under Delsarte duality}

In this section, we explore classes of codes that are either closed under duality or whose dual codes exhibit remarkable structural properties. We begin with MRD and near-MRD codes, proceed to a class of minimal codes, and conclude with a construction based on direct sums of MRD codes.

We can start with a classical result, that is,  the dual of an MRD code is an MRD code. In order to do so, we need the following result (whose proof is combinatorial and does not use any connection with the dual code). 

\begin{theorem}(\cite[Theorem 3.2]{zini2021scattered})\label{thm:MRDmaxhscatt}
    Assume that $h+1$ divides $km$ and let $n=km/(h+1)$.
    Let $\C$ be an $[n,k]_{q^m/q}$ code and let $U_{\C}$ be an associated system to $\C$.
    We have that $\C$ is an MRD code if and only if $U_{\C}$ is a maximum $h$-scattered subspace.
\end{theorem}

We can now prove the following classical result; see e.g. \cite{ga85a,de78}.

\begin{theorem}
    Let $\C$ be a non-degenerate $[n,k]_{q^m/q}$ code such that $\C^{\perp}$ is non-degenerate. 
    We have that $\C$ is an MRD code if and only if $\C^\perp$ is an MRD code.
\end{theorem}
\begin{proof}
    Denote by $d$ the minimum distance of $\C$ and let $h=m-d$. Note that, since $\C$ is non-degenerate, $h\leq m-2$.
    The code $\C$ is MRD if and only if
    \[
    h=\begin{cases}
        k-1 & \text{if } n\leq m,\\
        \frac{mk}n-1 & \text{if } n> m,
    \end{cases}
    \]
    and in the latter case we need that $n$ divides $mk$.
    In both the cases, we have that $h+1$ divides $km$ and if $n>m$ we have $n=km/(h+1)$.
    If $n\leq m$, then $\C$ is an MRD code if $d_1(\C)=n-k+1$ and hence by (\ref{eq:gen_weight_maximum_defect}) $s=1$ and $t_1=k$. By (\ref{eq:dualgeneralized weight}) we get $d_1(\C^\perp)=k+1$, i.e. $\C^\perp $ is an MRD code. the reverse implication can be established using a completely similar argument. Now, let $n>m$ and hence $n=km/(h+1)$.  By Theorem \ref{thm:MRDmaxhscatt}, the code $\C$ is MRD if and only if $U_{\C}$ is a maximum $h$-scattered subspace.
    By Corollaries \ref{cor:dualmaxhscatt} and \ref{cor:delsartedualitysubspaceandcode}, we have that $U_{\C}$ is a maximum $h$-scattered subspace if and only if $U_{\C^\perp}$ is a maximum $(m-h-2)$-scattered subspace.
    Again, by applying Theorem \ref{thm:MRDmaxhscatt}, we get that this is equivalent to ask that $\C^\perp$ is an MRD code.
\end{proof}

We now consider near MRD codes, which have been introduced in \cite{marino2023evasive}. 

\begin{definition}
    Let $\C$ be an $[n,k,d]_{q^m/q}$ code. We say that $\C$ is a \textbf{near MRD} code if $d=n-k$ and $d_r(\C)=n-k+r$ for every $r \in \{2,\ldots,k\}$.
\end{definition}

In \cite{marino2023evasive} the authors also proved a geometric characterization of these codes.

\begin{proposition}(\cite[Proposition 5.3]{marino2023evasive})\label{prop:charnearMRDsystem}
    Let $\C$ be a non-degenerate $[n,k]_{q^m/q}$ code. The code $\C$ is near MRD if and only if 
    \begin{itemize}
        \item $U_{\C}$ is $(k-2)$-scattered;
        \item $U_{\C}$ is not $(k-1)$-scattered;
        \item $U_{\C}$ is $(k-1,k)$-evasive.
    \end{itemize}
\end{proposition}

In other words, $\C$ is a near MRD code if and only if $U_{\C}$ is a $1$-defect subspace with sequence of maximum non-zero defects
\[ 0 < \varepsilon_{U_{\C}}(k-1)=1<\varepsilon_{U_{\C}}(k)=n-k. \]
We can now use Proposition \ref{prop:dual1-defect} to characterize the dual of a near MRD code.

\begin{proposition}
    Let $\C$ be a non-degenerate $[n,k]_{q^m/q}$ code such that $\C^{\perp}$ is non-degenerate. The code $\C$ is near MRD if and only if its dual is a near MRD code.
\end{proposition}
\begin{proof}
    By Proposition \ref{prop:charnearMRDsystem}, $\C$ is near MRD if and only if sequence of maximum non-zero defects of $U_{\C}$ is 
    \[ 0 < \varepsilon_{U_{\C}}(k-1)=1<\varepsilon_{U_{\C}}(k)=n-k. \]
    By applying  Proposition \ref{prop:dual1-defect} and Corollary \ref{cor:delsartedualitysubspaceandcode} to $U_{\C}$ we obtain that, this is equivalent to requiring that the sequence of maximum non-zero defects of $U_{\C^\perp}$ is 
    \[ 0 < \varepsilon_{U_{\C^\perp}}(n-k-1)=1<\varepsilon_{U_{\C^\perp}}(n-k)=k. \] 
    Therefore, we obtain the statement from Proposition \ref{prop:charnearMRDsystem}.
\end{proof}

Another relevant class of rank-metric codes is the family of linear quasi-MRD codes. 

The following provide a geometric characterization for some classes of codes which are quasi-MRD and/or dually quasi-MRD codes, extending \cite[Theorem 6.6]{marino2023evasive}.

\begin{theorem} \label{thm:quasi_MRD_k=3}
   Let $\C$ be a non-degenerate $[m+\rho,k]_{q^m/q}$ code with $1 \leq \rho < m/(k-1)$ such that $\C^{\perp}$ is non-degenerate. The code $\C$ is quasi-MRD if and only if $\varepsilon_{U_{\C}}(k-1)=\rho$. Whereas, the code $\C^\perp$ is quasi-MRD if and only if  $U_{\C}$ is a $(k-2)$-scattered space.
   Hence, $\C$ is dually quasi-MRD if and only if $U_{\C}$ is a $(k-2)$-scattered space and $\varepsilon_{U_{\C}}(k-1)=\rho$.
\end{theorem}
\begin{proof}
Let $U_{\C}$ be a system associated with $\C$ and note that since $\rho>0$, we have $\varepsilon_{U_{\C}}(k-1)>0$.
Hence, the length $s$ of the sequence of the maximum non-zero defects with respect to $U_{\C}$ is at least $2$. Also $U_{\C}$ is $(k-2)$-scattered if and only if $s=2$ and $t_1=k-1$.\\
Now, $\C^\perp$ is a non-degenerate $[m+\rho,m+\rho-k]_{q^m/q}$ code and 
 by (\ref{eq:dualgeneralized weight})
\[d(\C^\perp)=1+t_j\]
where $\varepsilon_{U_{\C}}(t_{j-1})<1\leq  \varepsilon_{U_{\C}}(t_{j})$, i.e. $d(\C^\perp)=1+t_1$. Also, $\C^\perp$ is quasi-MRD when 
\[d(\C^\perp)=m-\left\lceil \frac{m(m+\rho-k)}{m+\rho}\right\rceil+1=k.\]
Hence, $\C^\perp$ is a quasi-MRD if and only if $t_1=k-1$ and $s=2$, i.e. this happens if and only if $U_{\C}$ is a $(k-2)$-scattered space. 
Finally, $\C$ is a quasi-MRD code when 
\[d(\C)=m-\left\lceil \frac{km}{m+\rho}\right\rceil+1=m-k+1\]
and by Proposition \ref{prop:gener_weight_maximum_defect}, we have that $d(\C)=m+\rho-k+1-\varepsilon_{U_{\C}}(k-1)$, hence $\C$ is quasi-MRD if and only if $\varepsilon_{U_{\C}}(k-1)=\rho$. 
\end{proof}

In \cite{lia2024short}, the authors focused on the construction of $[m+2,3]_{q^m/q}$ minimal rank-metric codes (the shortest minimal rank-metric codes of dimension three), or equivalently on the construction of scattered subspaces in $\fqm^3$ of rank $m+2$.
They constructed scattered $\fq$-subspaces $U$ in $\fqm^3$ of dimension $m+2$, for certain values of $q$ and $m\geq 5$, with the following property: $U$ contains an $m$-dimensional $2$-scattered subspace of $\fqm^3$.

\begin{theorem}(\cite[Theorem 5.2]{lia2024short})\label{thm:boundweightLLMT}
    Let $U$ be an $(m+2)$-dimensional scattered $\fq$-subspace in $\fqm^3$ containing an $m$-dimensional $2$-scattered subspace of $\fqm^3$. 
    For any hyperplane $H$ in $\fqm^3$ we have
    \[ w_U(H) \in \{2,3,4\},\]
    i.e. $\varepsilon_U(H)\in \{0,1,2\}$
    and, for every $i \in \{2,3,4\}$, there exists at least one hyperplane having weight $i$.
\end{theorem}

As a consequence of the above result and Theorem \ref{thm:quasi_MRD_k=3}, we have a class of dually quasi-MRD codes.

\begin{corollary}
    Let $U$ be an $(m+2)$-dimensional scattered $\fq$-subspace in $\fqm^3$ containing an $m$-dimensional $2$-scattered subspace of $\fqm^3$. Any code associated with $U$ is a dually quasi-MRD code.
\end{corollary}

\begin{remark}
    Therefore, the results in \cite[Theorem 3.5 and Corollary 3.9]{lia2024short} guarantee the existence of dually quasi-MRD codes.
\end{remark}

The existence of non-trivial quasi-MRD codes (which are $\fqm$-linear) seems to be a widely open problem. In fact, some constructions of  $[7,4,3]_{q^5/q}$ quasi-MRD codes have been obtained in \cite[Corollary 6.9]{marino2023evasive} by using the construction of scattered subspaces in \cite[Theorem 5.1]{bartoli2021evasive}.

\begin{remark}
    As it happens for the matrix case, there are examples of quasi-MRD codes that are not dually quasi-MRD codes. Indeed, consider the following $5$-dimensional $\fq$-subspace of $\F_{q^4}^3$
    \[ U=\{(x^q,x^{q^2}-x,x^{q^3}-a\alpha)\colon x \in\F_{q^4},\alpha \in \fq\}, \]
    where $a \in \F_{q^4}$ such that $a^{q^2}\ne -a$. 
    As it has been proved in \cite[Case ($B_3$)]{bonoli2005fq}, $U$ is not scattered,  every hyperplane of $\F_{q^4}^3$ has weight at most three and there exist hyperplanes having weight three with respect to $U$. Therefore,   a code $\C$ associated with $U$ is a  $[5,3,2]_{q^4/q}$ code. By Theorem \ref{thm:quasi_MRD_k=3}, we have that $\C^{\perp}$ is not a quasi-MRD code and $\C$ is a quasi-MRD code as $\varepsilon_U(2)=1$.
\end{remark}

Another interesting class of rank-metric codes regards those that are direct sum of MRD codes. It is well known that the direct sum of MRD codes is not an MRD code; indeed, if $\C_1$ and $\C_2$ are two MRD codes in the direct sum $\C_1\oplus\C_2$ there are codewords of the form $(c,0)$ (or $(0,c)$) which have \emph{small} weight. 
Direct sums of one-dimensional MRD codes have been called \textbf{completely decomposable} rank-metric codes and have been studied in \cite{santonastaso2024completely}.
Also motivated by the connection with the problem of representability of uniform $q$-matroids that we will discuss in the next section, we introduce the following class of rank-metric codes. 

\begin{definition}
    Let ${\mathbf k}=(k_1,\ldots,k_t)$ and  ${\mathbf n}=(n_1,\ldots,n_t)$, with $k_i<n_i\leq m$ for all $i\in [t]$. 
    Let $N=n_1+\ldots+n_t$ and $K=k_1+\ldots+k_t$.
    An $[N,K]_{q^m/q}$ code $\C$ is said to be an $(\mathbf{n},\mathbf{k})$-\textbf{MRD code} if there exist $t$ MRD codes $\C_1,\ldots,\C_t$ with parameters $[n_1,k_1]_{q^m/q},\ldots,[n_t,k_t]_{q^m/q}$, respectively, such that
    \[\C=\C_1\oplus \ldots\oplus\C_t,\]
    and with the property that every codeword $c =(c_1,c_2,\dots, c_t)$ with $c_i \in \C_i\setminus \{\underline 0\}$ for $i \in [t]$  has weight at least $N-K+1$.
\end{definition}

\begin{remark}\label{rem:smallweightcodewords(n,k)-MRD}
    Let $\C$ be an $(\mathbf{n},\mathbf{k})$-{MRD code} as in the above definition. 
    It is worth mentioning that the codewords with weight less than $N-K+1$ are of the form $(c_1,\ldots,c_t)$ where at least one of the $c_i$ is  zero.
\end{remark}

The geometric counterparts of these codes are subspaces that are $\mathbf{k}$-scattered with respect to hyperplanes.

\begin{proposition}\label{prop:geo(n.k)-MRD}
    Let $\C$ be an $[N,K]_{q^m/q}$ code and let $U_{\C}$ be a system associated with $\C$.
    Let ${\mathbf k}=(k_1,\ldots,k_t)$, ${\mathbf n}=(n_1,\ldots,n_t)$, $N=n_1+\ldots+n_t$ and $K=k_1+\ldots+k_t$ with $k_i<n_i\leq m$ for every $i\in \{1,\ldots,t\}$.
    The code $\C$ is $(\mathbf{n},\mathbf{k})$-MRD code if and only if $U_{\C}$ is $\mathbf{k}$-scattered with respect to hyperplanes of type $\mathbf{n}$. 
\end{proposition}
\begin{proof}
    Assume that $\C$ is $(\mathbf{n},\mathbf{k})$-MRD code.
    This means that $\C$ admits a generator matrix in block form
    \begin{equation}\label{eq:Ggenmat}
    G=\begin{pmatrix}
        G_1 & 0 & \cdots & 0\\
        0 & G_2 & \cdots & 0\\
        \vdots & & \vdots\\
        0 & \cdots & \cdots & G_t
    \end{pmatrix},
    \end{equation}
    where $G_i \in \fqm^{k_i\times n_i}$ is a generator matrix of an MRD with parameters $[n_i,k_i]_{q^m/q}$.
    Let $F_i$ be the $\fqm$-span of the $n_i$ columns in the $i$-th block and let $U_i$ be the $\fq$-span of the $n_i$ columns in the $i$-th block, for every $i \in \{1,\ldots,t\}$.
    Observe that $\fqm^K=F_1\oplus\ldots\oplus F_t$ and $U_{\C}=U_1\oplus\ldots\oplus U_t$.
    Since the $G_i$'s are generator matrices of MRD codes, by Theorem \ref{thm:MRDmaxhscatt} the $U_i$'s are scattered subspaces with respect to the hyperplanes in $F_i$.
    It only remains to prove that  $w_{U_{\C}}(H)\leq K-1$ for each hyperplane $H$ such that $F_i\not\subseteq H$ for $i \in [t]$.
    By contradiction, suppose that there exists a hyperplane $H$ not containing any of the $F_i$'s such that 
    $w_{U_{\C}}(H)\geq K$.
    This would mean that, by Theorem \ref{th:connection}, there exists a codeword  $c=( c_1,c_2,\dots c_t)$ with $c_i \in \C_i\setminus \{\underline 0\}$ for $i \in [t]$  having weight at most $N-K$, a contradiction.
Conversely, suppose that $U_{\C} \subseteq \fqm^K$ is $\mathbf{k}$-scattered with respect to hyperplanes of type $\mathbf{n}$. Up to the action of $\mathrm{GL}(K,q^m)$, we may assume that $U_{\C}=U_1\oplus \ldots\oplus U_t$ where $U_i$ is an $\fq$-subspace of $\fqm^{k_i}$ and $\langle U_i \rangle_{\fqm}=\fqm^{k_i}$ for any $i$. We can reverse the above arguments by considering a matrix $G$ as in \ref{eq:Ggenmat} whose columns $\fq$-span $U_{\C}$.
\end{proof}

This geometric view allows us to prove that the class of $(\mathbf{n},\mathbf{k})$-MRD code is closed under duality.

\begin{theorem}\label{thm:dual(n,k)MRD}
    Let $\C$ be an $[N,K]_{q^m/q}$ code and consider ${\mathbf k}=(k_1,\ldots,k_t)$, ${\mathbf n}=(n_1,\ldots,n_t)$, $N=n_1+\ldots+n_t$ and $K=k_1+\ldots+k_t$, with $k_i<n_i\leq m$ for all $i\in \{1,\ldots,t\}$. 
    The code $\C$ is an $(\mathbf{n},\mathbf{k})$-MRD code if and only if $\C^\perp$ is an $(\mathbf{n},\mathbf{n}-\mathbf{k})$-MRD code.
\end{theorem}
\begin{proof}
    By Proposition \ref{prop:geo(n.k)-MRD}, $\C$ is an $(\mathbf{n},\mathbf{k})$-MRD code if and only if an associated system $U_{\C}$ is $\mathbf{k}$-scattered with respect to hyperplanes of type $\mathbf{n}$.
    By Theorem \ref{thm:dual(n,k)MRD}, the Delsarte dual $U_{\C}^d$ of $U_{\C}$ is $(\mathbf{n}-\mathbf{k})$-scattered with respect to hyperplanes of type $\mathbf{n}$. By Corollary \ref{cor:delsartedualitysubspaceandcode}, $U_{\C}^d=U_{\C^\perp}$, and therefore we get the assertion by using again Proposition \ref{prop:geo(n.k)-MRD} on $U_{\C^\perp}$.
\end{proof}

This is a very interesting property that allows us to prove that, if $t=2$, all the $(\mathbf{n},\mathbf{k})$-MRD codes have the same weight distribution. 

\begin{theorem}\label{thm:weightdistribution}
Let ${\mathbf k}=(k_1,k_2)$ and  ${\mathbf n}=(n_1,n_2)$, with $k_i<n_i\leq m$ for all $i\in \{1,2\}$. The weight distribution of an $(\mathbf{n},\mathbf{k})$-MRD code depends only on $\mathbf{n},\mathbf{k},q$ and $m$.
\end{theorem}
\begin{proof}
Let $\C\subseteq \F_{q^m}^{n_1+n_2}$ be an $(\mathbf{n},\mathbf{k})$-MRD code. By Theorem \ref{thm:dual(n,k)MRD}, its dual $\C^\perp$ is a $(\mathbf{n},\mathbf{n}-\mathbf{k})$-MRD code. Let $N=n_1+n_2$, $K=k_1+k_t$.
Since $\C_1, \C_1^\perp,\C_2,\C_2^\perp$ are all MRD codes, their weight distribution is then uniquely determined; see e.g. \cite[Remark 45]{ravagnani2016rank}. 
Also,  for all $i\in \{1,\ldots,N-K\}$ we have 
$$A_i:=A_i(\C)=A_i(\C_1)+A_i(\C_2),$$
 that means that they are completely determined,
and 
$$B_j:=A_j(\C^\perp)=A_j(\C_1^\perp)+A_j(\C_2^\perp)$$
for all $j\in \{1,\ldots,K\}$, that means that they are completely determined.
 Moreover, by \cite[Proposition 4.]{gadouleau2008macwilliams}, $$\sum_{i=0}^{N-v}B_i\binom{N-i}{v}_q=q^{m(N-K-v)}\sum_{j=0}^vA_j\binom{N-j}{v-j}_q$$
for all $v\in\{0,\ldots,N\}$. 

Let $v=N-K+r$. 
$$\sum_{j=N-K+1}^{N-K+r}A_j\binom{N-j}{K-r}_q
=q^{mr}\sum_{i=0}^{K-r}B_i\binom{N-i}{N-K+r}_q-\sum_{j=0}^{N-K}A_j\binom{N-j}{K-r}_q$$
for all $r\in \{1,\ldots,K\}$. The right hand-side of the system is known, whereas the left hand-side is triangular (with non-zero diagonal), so that the solution is uniquely determined.
\end{proof}

\begin{remark}
Let us remark that the proof of Theorem \ref{thm:weightdistribution} also provides a way to explicitly compute the weight distribution, by recurrence. Actually, with the notations above, for all $r\in \{1,\ldots,K\}$, the following holds:
\begin{equation}\label{eq:formAirelBj}
A_{N-K+r}
=q^{rm}\sum_{i=0}^{K-r}B_i\binom{N-i}{N-K+r}_q-\sum_{j=0}^{N-K+r-1}A_j\binom{N-j}{K-r}_q.\end{equation}
\end{remark}

We now derive some necessary conditions on the existence of $((n_1,n_2)),(k_1,k_2))$-MRD codes.
We start with the case $(k_1,k_2)=(1,1)$.

\begin{theorem}\label{thm:case(1,1)bound}
    Let ${\mathbf k}=(1,1)$ and  ${\mathbf n}=(n_1,n_2)$, with $1<n_i\leq m$ for all $i\in \{1,2\}$.
    Suppose that an $((n_1,n_2)),(1,1))$-MRD code $\C$ exists in $\fqm^{n_1+n_2}$.
    We have that $n_i\leq m/2$ for every $i \in \{1,2\}$. Moreover,
    \begin{itemize}
        \item if $n_1=n_2$, then $m\geq n_1+n_2$ and equality holds if and only if $n_1=n_2=m/2$;
        \item if $n_1\ne n_2$, then $m\geq n_1+n_2+1$.
    \end{itemize}
\end{theorem}
\begin{proof}
    The $q$-system $U_{\C}$ associated with $\C$ is $(1,1)$-scattered with respect to the hyperplanes of type $(n_1,n_2)$ in $\fqm^2$, because of Proposition \ref{prop:geo(n.k)-MRD}.
    By \cite[Theorem 4.4]{napolitano2022linear}, we have that $n_1\leq m/2$ and $n_2\leq m/2$, hence $m\geq n_1+n_2$.
    Moreover, the equality holds if and only if $n_1=n_2=m/2$.
\end{proof}

In particular, we note that the rank-metric codes associated with the construction in \cite[Corollary 4.7]{napolitano2022linear} give examples of $((m/2,m/2)),(1,1))$-MRD codes, showing that the above bound is sharp when $n_1=n_2$.

\begin{remark}
Using the formulas \eqref{eq:formAirelBj}, one can obtain non-existence results for certain parameters. This is possible because for extensions that are too small, the resulting weight distributions include negative values, which is clearly not feasible. Another criterion for non-existence is that positive coefficients appear where, due to maximal rank considerations, this is not possible.
Unfortunately, obtaining a general result appears to be too complex. Therefore, we provide Table \ref{tab:dati} with some small cases. 
Table \ref{tab:dati} summarizes all the other results we obtained putting in the first four columns the lengths and dimensions and in the last column the bound on $m$ that we get. 
\end{remark}

We illustrate one of the cases considered in Table \ref{tab:dati} in the following example.

\begin{example}
Let $n_1=n_2=3$, $k_1=1$, $k_2=2$. In this case, 
$$A_0(\C_1)=1, \ A_1(\C_1)=A_2(\C_1)=0, \ A_3(\C_1)=q^m-1$$
and
$$A_0(\C_2)=1, \ A_1(\C_2)=0, \ A_2(\C_2)=(q^m-1)(q^2+q+1), \ A_3(\C_2)=(q^m-1)(q^m-q^2-q).$$
Hence, by the above formula
\begin{align*}
A_0(\C) & = 1\\
A_1(\C) & = 0\\
A_2(\C) & = (q^m-1)(q^2+q+1) \\
A_3(\C) & = (q^m-1)(q^m-q^2-q+1) \\
A_4(\C) & = (q^m-1)(q-1)(q+1)(q^2-q+1)(q^2+q+1)^2\\
A_5(\C) & = (q^m-1)q(q^2+q+1)(q^{m+2}-q^6-q^5+1) \\
A_6(\C) & = (q^m-1)(q^{2m}-q^{m+5}-q^{m+4}-q^{m+3}+q^{9}+q^{8}+q^{7}-q^3)
\end{align*}
Note that $A_5$ is negative iff $m\leq 4$ and $A_i\geq 0$ for $i\in\{0,\ldots,6\}$ if $m\geq 5$. However, $A_6\neq 0$ if $m=5$, which is not possible, since the maximum possible rank is $5$.  
\end{example}

\begin{table}[h!]
\centering
\begin{tabular}{cccccc}
\hline
$n_1$ & $n_2$ & $k_1$ & $k_2$ & $m$ \\
\hline
3 & 3 & 1 & 2 & $\geq$ 6 \\
4 & 4 & 1 & 2 & $\geq$ 8 \\
4 & 4 & 1 & 3 & $\geq$ 8 \\
4 & 4 & 2 & 2 & $\geq$ 8 \\
5 & 5 & 1 & 2 & $\geq$ 10 \\
5 & 5 & 1 & 3 & $\geq$ 10 \\
5 & 5 & 1 & 4 & $\geq$ 10 \\
5 & 5 & 2 & 2 & $\geq$ 10 \\
5 & 5 & 2 & 3 & $\geq$ 10 \\
\hline
\end{tabular}
\caption{Lower bounds on $m$ for $((n_1,n_2),(k_1,k_2))$-MRD codes in $\fqm^{n_1+n_2}$.}
\label{tab:dati}
\end{table}

\begin{question}
Looking at Table \ref{tab:dati} and Theorem \ref{thm:case(1,1)bound}, it is natural to wonder whether the bound 
$m\geq n_1+n_2$ holds in general (for $k_i\in \{1,\ldots,n_i-1\}$). In particular, is it true that $A_{n_1+n_2}(\C)>0$ for $m=n_1+n_2-1$?
\end{question}

\section{Consequences on the representability of uniform $q$-matroids}

The notion of $q$-matroids originates from Crapo’s PhD thesis \cite{crapo1964theory}. In more recent years, $q$-matroids have been revisited in \cite{jurrius2018defining} and have attracted considerable interest due to their connection with rank-metric codes. See, for example, \cite{gorla2019rank, shiromoto2019codes, ghorpade2020polymatroid, byrne2022constructions, byrne2021weighted, gluesing2021q}.

\begin{definition}
    A \textbf{$q$-matroid with ground space} $E$ is a pair $\mM=(E,\rho)$, where $E$ is a finite dimensional $\F_q$-vector space and $\rho$ is an integer-valued 
	function defined on $\mL(E)$ with the following properties:
	\begin{itemize}
		\item[(R1)] Boundedness: $0\leq \rho(A) \leq \dim A$, for all $A\in \mL(E)$. 
		\item[(R2)] Monotonicity: $A\leq B \Rightarrow \rho(A)\leq \rho(B)$, for all $A,B\in\mL(E)$. 
		\item[(R3)] Submodularity: $\rho(A+ B)+\rho(A\cap B)\leq \rho(A)+\rho(B)$, for all $A,B\in\mL(E)$.  
	\end{itemize}
The function $\rho$ is called \textbf{rank function} and the value $\rho(\mM) := \rho(E)$ is the \textbf{rank of the $q$-matroid}. The value $h(\mM)=\dim_{\F_q}(E)$ is the \textbf{height of $\mM$}.
\end{definition}

A well-known family of $q$-matroids is the family of \emph{uniform $q$-matroids}; see \cite{jurrius2018defining}.

\begin{definition}\label{def:uniform}
Let $0\leq k\leq n$. For each $V\in\mL(\F_q^n)$, define $\rho(V) :=\min\{k,\dim(V)\}$. Then $(\F_q^n,\rho)$ is a $q$-matroid. It is called the \textbf{uniform $q$-matroid on $\F_q^n$ of rank $k$} and is denoted by $\mU_{k,n}(q)$.
\end{definition}

A possible way to construct $q$-matroids is via matrices; see \cite{jurrius2018defining, gorla2019rank}. Let $G$ be a $k \times n$ matrix with entries in $\F_{q^m}$ and for every $U\in\mL(\F_q^n)$, let $A^U$ be a matrix whose columns form a basis of $U$. Then the map 
\begin{equation}\label{eq:rank1}
    \rho:\mL(\F_q^n) \to \mathbb{Z}, \  U\mapsto \rk(G A^U),
\end{equation}
is the rank function of a $q$-matroid, which we denote by $\mM_G$ and we call the \textbf{$q$-matroid represented by $G$}. 

A $q$-matroid $\mM$ with ground space $\F_q^n$ is called \textbf{$\F_{q^m}$-representable} if $\mM=\mM_G$ for some matrix $G$ with  $n$ columns and entries in $\F_{q^m}$, whose rank equals the rank of $\mM$. The $q$-matroid $\mM$ is called \textbf{representable} if it is \textbf{$\F_{q^m}$-representable} over some field extension $\F_{q^m}$.
In other words, we can say that $\mM$ is $\F_{q^m}$-representable if there is a rank-metric code with generator matrix $G$, such that $\mM=\mM_G$. 
As proved in \cite{alfarano2022cyclic}, the representability of a $q$-matroid can be expressed in terms of $q$-systems, for further details we refer to \cite[Theorem 4.6]{alfarano2022cyclic}.

\begin{remark}
Let $\mU_{k,n}(q)$ be the uniform $q$-matroid of rank $k$ with ground space $\F_q^n$. Then the \emph{trivial} uniform $q$-matroid $\mU_{0,n}(q)$ and the \emph{free} uniform $q$-matroid $\mU_{n,n}$ are
representable over $\F_q$. For $0 < k < n$, the uniform $q$-matroid $\mU_{k,n}(q)$ is representable over $\F_{q^m}$ if and only if
$m \geq n$.  
Analogously, in the language of $q$-systems, a $\Fm$-representation of $\mathcal U_{k,n}(q)$ is given by any $\Fmk$ system which is $(k-1)$-scattered, and this is known to exist if and only if $m\ge n$; see e.g. \cite{de78,sheekeyVdV}.
\end{remark}

In \cite{ceria2021direct} the notion of direct sum of two $q$-matroids has been introduced. We are going to define the direct sum of $t$ many $q$-matroids iteratively, thanks to the associativity property of the direct sum, proved in \cite{gluesing2023decompositions}.
To this aim we introduce the following map.
For $n_1,n_2,\ldots, n_t\in\mathbb{N}$ and $n=n_1+n_2+\cdots+n_t$, consider $E_1,\ldots,E_t$ $\fq$-subspaces of an $\fq$-vector space $E$ such that $E=E_1\oplus\cdots\oplus E_t$. We denote by
$\pi_i:E\longrightarrow E_i$ is the projection onto $E_i$, for any $i \in [t]$.

\begin{definition}
    Let $\mM_i=(E_i,\rho_i),\,i\in[t],$ be $q$-matroids and set $E=E_1\oplus\cdots\oplus E_t$.
Define $\rho'_i:\mL(E)\longrightarrow \mathbb{N}_0,\ V\longmapsto \rho_i(\pi_i(V))$ for $i\in[t]$.
Then $\mM'_i=(E,\rho'_i)$ is a $q$-matroid. 
We define the \textbf{direct sum of $\mM_1, \ldots,\mM_t$} to be the $q$-matroid $\mM:=(E,\rho)$, where $\rho$ is defined iteratively as follows:
\begin{itemize}
\item If $t=2$, then
\begin{equation}\label{e-rho}
\rho:\mL(E)\longrightarrow\mathbb{N}_0,\quad V\longmapsto \dim V+\min_{X\leq V}\big(\rho'_1(X)+\rho'_2(X)-\dim X\big).
\end{equation}
\item If $t>2$, then $\rho$ is the rank function of $(\mM_1\oplus\cdots\oplus\mM_{t-1})\oplus \mM_t$.
\end{itemize}
\end{definition}

The representability of the direct sum of uniform $q$-matroids has been characterized geometrically as follows; see \cite[Theorem 2.4]{alfarano2024representability}.

\begin{theorem}\label{thm:charact_independence}
Let $k_1,\ldots,k_t,k,n_1,\ldots,n_t, n, m$ be positive integers, with $1\le k_i< n_i\le m$ for $i \in [t]$ and $k=k_1+\ldots+k_t$, $n=n_1+\ldots+n_t$.
Then the following statements are equivalent.
\begin{enumerate}
    \item $\mU_{k_1,n_1}(q)\oplus\ldots\oplus\mU_{k_t,n_t}(q)$ is $\Fm$-representable;
    \item there exists a $\mathbf{k}$-scattered subspace with respect to the hyperplanes of type $\mathbf{n}=(n_1,\ldots,n_t)$, where $\mathbf{k}=(k_1,\ldots,k_t)$;
    \item there exists an $(\mathbf{n},\mathbf{k})$-MRD code in $\fqm^n$.
\end{enumerate}
\end{theorem}
\begin{proof}
    \cite[Theorem 2.4]{alfarano2024representability} ensures that (1) and (2) are equivalent. 
    The remaining equivalences immediately follow from Proposition \ref{prop:geo(n.k)-MRD}.
\end{proof}

As a consequence, by constructing certain subspaces that are $\mathbf{k}$-scattered with respect to the hyperplanes, in \cite{alfarano2024representability} the authors derived the following result on the representability of the direct sum of uniform $q$-matroids.

\begin{theorem}\label{thm:representability_direct_sum}(\cite[Theorem 3.2]{alfarano2024representability})
    Let $k_1,\ldots, k_t, n_1, \ldots n_t$ be positive integers such that $k_i< n_i$ for each $i \in [t]$. Then, the direct sum
    $$\mU_{k_1,n_1}(q) \oplus \ldots \oplus \mU_{k_t,n_t}(q)$$
    is representable. In particular, it is represantable over every field $\Fm$ such that $m=m_1\cdot\ldots\cdot m_t$ for any $m_1,\ldots,m_t$ satisfying $\gcd(m_i,m_j)=1$ for each $i \neq j$ and $m_i\geq n_i$ for $i\in [t]$.
\end{theorem}

The natural question now is to determine on which extension of $\fq$ the direct sum of uniform $q$-matroids is representable.
We are not going to completely answer this question, but we will give some insight on that.
We start by using the fact that property of being $\mathbf{k}$-scattered with respect to the hyperplanes is closed under Delsarte duality.

\begin{corollary}\label{cor:dualdirectsumuniformqmat}
Let $k_1,\ldots,k_t,k,n_1,\ldots,n_t, n, m$ be positive integers, with $1\le k_i< n_i\le m$ for $i \in [t]$ and $k=k_1+\ldots+k_t$, $n=n_1+\ldots+n_t$.
Then the following statements are equivalent.
\begin{enumerate}
    \item $\mU_{k_1,n_1}(q)\oplus\ldots\oplus\mU_{k_t,n_t}(q)$ is $\Fm$-representable;
    \item $\mU_{n_1-k_1,n_1}(q)\oplus\ldots\oplus\mU_{n_t-k_t,n_t}(q)$ is $\Fm$-representable.
\end{enumerate}
\end{corollary}
\begin{proof}
    By Theorem \ref{thm:charact_independence}, $\mU_{k_1,n_1}(q)\oplus\ldots\oplus\mU_{k_t,n_t}(q)$ is $\Fm$-representable if and only if for any $i \in [t]$ there exists an $\fq$-subspace $\mathcal{S}_i$ of dimension $n_i$ in $\Fm^{k_i}$ such that $\mathcal{S}=\mS_1\oplus\ldots\oplus\mS_t$  is $\mathbf{k}$-scattered with respect to the hyperplanes, where $\mathbf{k}=(k_1,\ldots,k_t)$.
    By Theorem \ref{thm:MRDmaxhscatt}, this happens if and only if $\mS^d$ is $(\mathbf{n}-\mathbf{k})$-scattered with respect to the hyperplanes. Theorem \ref{thm:representability_direct_sum} guarantees the assertion.
\end{proof}

This corollary can help us to obtain other instances of the determination of the field extension where the direct sum is representable.
For instance, for the direct sum of two uniform $q$-matroids of rank $1$ we have the following.

\begin{theorem}\label{thm:summary1+1}(\cite[Theorem 4.4]{alfarano2024representability})
    Let $n_1,n_2\geq 2$ be two positive integers. The direct sum
    $$\mU_{1,n_1}(q)\oplus \mU_{1,n_2}(q)$$
    is $\Fm$-representable if at least one of the following holds.
    \begin{enumerate}
       \item $m$ is even and $m\geq 2\max\{n_1,n_2\}$; 
       \item $m\geq n_1n_2$;
       \item $m=t_1t_2$ with $t_1\ge n_1$ and $n_2\le \frac{t_1(t_2-1)}{2}+1$;
       \item $m=t_1t_2$ with $t_1\geq n_1,n_2$ and $t_2\geq 2$;
       \item $q=p^h$, $m=p^r$, $n_1+n_2-1\le \frac{m}{2}$.
    \end{enumerate}
\end{theorem}

As a consequence of the above result and Corollary \ref{cor:dualdirectsumuniformqmat} we have the following.

\begin{corollary}
    Let $n_1,n_2\geq 2$ be two positive integers. The direct sum
    $$\mU_{n_1-1,n_1}(q)\oplus \mU_{n_2-1,n_2}(q)$$
    is $\Fm$-representable if at least one of the following holds.
    \begin{enumerate}
       \item $m$ is even and $m\geq 2\max\{n_1,n_2\}$;  
       \item $m\geq n_1n_2$;
       \item $m=t_1t_2$ with $t_1\ge n_1$ and $n_2\le \frac{t_1(t_2-1)}{2}+1$; 
       \item $m=t_1t_2$ with $t_1\geq n_1,n_2$ and $t_2\geq 2$; 
       \item $q=p^h$, $m=p^r$, $n_1+n_2-1\le \frac{m}{2}$.
    \end{enumerate}
\end{corollary}

A classical function which is associated with a $q$-matroid is the rank generating function; see \cite{shiromoto2019codes}.

\begin{definition}
    Let $\mathcal{M}=(E,\rho)$ be a $q$-matroid. 
    The \textbf{rank generating function} of $\mathcal{M}$ is defined as
    \[ R_{\mathcal{M}}(X_1,X_2,X_3,X_4)=\sum_{D \in \mathcal{L}(E)} f_\mathcal{M}^D(X_1,X_2)f^{\dim(D)}(X_3,X_4), \]
    where  
    \[f_\mathcal{M}^J(X_1,X_2)=X_1^{\rho(E)-\rho(J)}X_2^{\dim(J)-\rho(J)},\]
    for any subspace $J$ of $E$, and
    \[ g^l(X_3,X_4)=\prod_{i=0}^{l-1} (X_3-q^iX_4), \]
    for any nonnegative integer $l \in \N_0$.
\end{definition}

When the $q$-matroid is representable, there is a relation between the rank generating function of the $q$-matroid and the weight enumerator of the involved rank-metric code.

\begin{theorem}(\cite[Theorem 14]{shiromoto2019codes})
    Let $\C$ be an $[n,k]_{q^m/q}$ code and let $G$ be an its generator matrix.
    We have that
    \[ W_{\C}(X,Y)=Y^{n-k}R_{\mathcal{M}_G}\left(qY^{1/m},\frac{1}{Y^{1/m}},X,Y \right). \]
\end{theorem}

As a consequence of the above result and Theorem \ref{thm:weightdistribution}, we have that if the direct sum of uniform $q$-matroids is representable, then its rank generating function depends only on the involved parameters.

\begin{corollary}
    Let $k_1,k_2,k,n_1,n_2, n, m$ be positive integers, with $1\le k_i< n_i\le m$ for $i \in [2]$ and $k=k_1+k_2$, $n=n_1+n_2$.
Assume that $\mU_{k_1,n_1}(q)\oplus\mU_{k_2,n_2}(q)$ is $\Fm$-representable. We have that the rank generating function of $\mU_{k_1,n_1}(q)\oplus\mU_{k_2,n_2}(q)$ depends only on $k_1,k_2,k,n_1,n_2,m$ and $q$.
\end{corollary}

Using the consequences of Theorem \ref{thm:weightdistribution}, we derive some bounds on the minimum degree of the field extension over which is representable. Indeed, using Theorem \ref{thm:case(1,1)bound} we obtain the following.

\begin{corollary}
    Let $n_1,n_2, m$ be positive integers, with $1< n_i\le m$ for $i \in [2]$.
Assume that $\mU_{1,n_1}(q)\oplus\mU_{1,n_2}(q)$ is $\Fm$-representable, then 
    \begin{itemize}
        \item if $n_1=n_2$, then $m\geq n_1+n_2$ and equality holds if and only if $n_1=n_2=m/2$;
        \item if $n_1\ne n_2$, then $m\geq n_1+n_2+1$.
    \end{itemize}
\end{corollary}

More bounds can be obtained using Theorem \ref{thm:weightdistribution} and Table \ref{tab:dati}.

\begin{remark}
    The bounds provided in Table \ref{tab:dati}, can also be read in the context of representability of $\mathcal{U}_{k_1,n_1}\oplus\mathcal{U}_{k_2,n_2}$ over $\fqm$. For example, $\mathcal{U}_{1,3}\oplus\mathcal{U}_{2,3}$ is not $\fqm$-representable for any $m\leq 5$.  
\end{remark}

\bigskip

\section*{Acknowledgments}
This research was also supported by Bando Galileo 2024 – G24-216. The first author is partially supported by the ANR-21-CE39-0009 - BARRACUDA (French \emph{Agence Nationale de la Recherche}).
The last two authors were partially supported by the Italian National Group for Algebraic and Geometric Structures and their Applications (GNSAGA - INdAM).

\bibliographystyle{abbrv}
\bibliography{biblio}

\end{document}